\pdfoutput=1
\RequirePackage{ifpdf}
\ifpdf 
\documentclass[pdftex]{sigma}
\else
\documentclass{sigma}
\fi

\numberwithin{equation}{section}

\newtheorem{Theorem}{Theorem}[section]
\newtheorem*{Theorem*}{Theorem}
\newtheorem{Corollary}[Theorem]{Corollary}
\newtheorem{Lemma}[Theorem]{Lemma}
\newtheorem{Proposition}[Theorem]{Proposition}
 { \theoremstyle{definition}
\newtheorem{Definition}[Theorem]{Definition}

\newtheorem{Remark}[Theorem]{Remark} }
 \usepackage{physics}

\begin{document}
\allowdisplaybreaks

\newcommand{\arXivNumber}{2106.03421}

\renewcommand{\PaperNumber}{014}

\FirstPageHeading

\ShortArticleName{$q$-Selberg Integrals and Koornwinder Polynomials}

\ArticleName{$\boldsymbol{q}$-Selberg Integrals and Koornwinder Polynomials}

\Author{Jyoichi KANEKO}

\AuthorNameForHeading{J.~Kaneko}

\Address{Department of Mathematical Sciences, University of the Ryukyus,\\
 Nishihara, Okinawa 903-0213, Japan}
\Email{\href{mailto:kaneko@math.u-ryukyu.ac.jp}{kaneko@math.u-ryukyu.ac.jp}}

\ArticleDates{Received June 23, 2021, in final form February 14, 2022; Published online February 28, 2022}

\Abstract{We prove a generalization of the $q$-Selberg integral evaluation formula. The integrand is that of $q$-Selberg integral multiplied by a factor of the same form with respect to part of the variables. The proof relies on the quadratic norm formula of Koornwinder polynomials. We also derive generalizations of Mehta's integral formula as limit cases of our integral.}


\Classification{33D52; 05A30; 11B65}\vspace{-1mm}

\section{Introduction}
 The Selberg integral is one of the most important multiple integrals which admits explicit evaluation \cite{Se},
 and once rediscovered around 1980 there has been followed a flood of studies on it and its various
 generalizations, see \cite{FW} and references therein. In particular, Askey \cite{As} proposed a $q$-analogue of
 the Selberg integral with its conjectural explicit evaluation, see~\eqref{eq1.1} below, and this conjecture was
 proved by Habsieger \cite{Hab} and Kadell \cite{Kad} in entirely different ways.
 Aomoto~\cite{Ao} proposed and evaluated $q$-Selberg type integrals attached to any reduced root systems.
Let us mention some recent relevant papers: other proofs of the Habsieger--Kadell's formula \cite{GLXZ,Kar,XZ}
and variations and extensions \cite{IF, W1,W2}. To state the formula we require some definitions.
We~assume that $0<q<1$ and let\vspace{-1mm}
\begin{equation*}
 ( x;q)_{\infty}=(x)_{\infty}=\prod_{i=0}^{\infty} \big(1-xq^{i}\big), \qquad (x;q)_a=(x)_{a}=\frac{(x;q)_{\infty}}{(x;q^{a})_{\infty}}, \qquad
 [x]_q =\frac{1-q^x}{1-q},\vspace{-1mm}
 \end{equation*}
 and the $q$-gamma function for $x\neq 0, -1,\dots $\vspace{-1mm}
 \begin{gather*}
 \Gamma_q(x)=\frac{(q)_{\infty}}{(q^x)_{\infty}} (1-q)^{1-x}.\vspace{-1mm}
\end{gather*}
We shall write $(x_1,\dots ,x_m;q)_{\infty}$ for $\prod_{i=1}^{m}(x_i;q)_{\infty}$.
The $q$-integral of a function $f(t_1,\dots,t_n)$ on~$[0,1]^n$ is defined by\vspace{-1mm}
\begin{gather*}
 \int_{[0,1]^n} f(t_1,\dots,t_n) {\rm d}_qt_1 \cdots {\rm d}_qt_n =(1-q)^n \sum_{j_1=0,\dots,j_n=0}^{\infty}
 f\big(q^{j_1},\dots,q^{j_n}\big)q^{j_1+\cdots +j_n}.\vspace{-1mm}
 \end{gather*}
 Then the Habsieger--Kadell's integral formula is
 \begin{gather}
\int_{[0,1]^n}
 \prod_{i=1}^n t_i^{\alpha-1}\frac{(qt_i)_{\infty}}{(q^{\beta}t_i)_{\infty}}
 \prod_{1\le i<j\le n}
t_i^{2k}\bigg(q^{1-k}\frac{t_j}{t_i}\bigg)_{2k}\, {\rm d}_qt_1\cdots {\rm d}_qt_n\nonumber
 \\ \qquad
 {} =q^{\alpha k\binom{n}{2}+2k^2\binom{n}{3}}\prod_{i=1}^{n}\frac{\Gamma_q(1+ik)\Gamma_q(\alpha +(i-1)k)
 \Gamma_q(\beta+(i-1)k)} {\Gamma_q(1+k)\Gamma_q(\alpha+\beta+(n+i-2)k)},
 \label{eq1.1}
 \end{gather}
 for $\operatorname{Re}\alpha, \operatorname{Re} \beta >0$, $k\in \mathbb{Z}_{\ge 0}$. In the present paper
 we prove a generalization of this $q$-Selberg integral formula. Namely we evaluate a $q$-integral which
 has the integrand of the $q$-Selberg integral multiplied by a factor of the same form with respect to
 part of the variables:
\begin{gather*} 
\int_{[0,1]^n} \prod_{i=n_0+1}^n t_i\big(1-q^{\beta}t_i\big) \prod_{n_0+1\le i<j\le n}\big(t_i-q^{-k}t_j\big)\big(t_i-q^{k+1}t_j\big)
 \prod_{i=1}^n t_i^{\alpha-1}\frac{(qt_i)_{\infty}}{\big(q^{\beta}t_i\big)_{\infty}}
 \\ \qquad
{}\times \prod_{1\le i<j\le n}
t_i^{2k}\bigg(q^{1-k}\frac{t_j}{t_i}\bigg)_{2k}\, {\rm d}_qt_1\cdots {\rm d}_qt_n,
 \end{gather*}
or a factor similar to this,
see~\eqref{eq7.2} and~\eqref{eq7.4} in Theorem~\ref{theorem7.3} for the precise forms of the evaluation
in terms of $q$-gamma function. This formula gives, by taking a limit, the evaluation of certain generalizations
of the so-called Mehta integral (see Corollaries~\ref{corollary7.7} and~\ref{corollary7.11}).

The motivation to study this type of integral stems from the Baker--Forrester constant term
(ex-)conjecture \cite{BF}: Suppose $a,b,k\in \mathbb Z_{\ge 0}$ and let
\begin{gather*}
M_n(a,b,k;q)=\prod_{l=0}^{n-1}
\frac{\Gamma_q(a+b+1+kl)\Gamma_q(1+k(l+1))}{\Gamma_q(a+1+kl)\Gamma_q(b+1+kl)\Gamma_q(1+k)}.
\end{gather*}
Then
\begin{gather}
\mbox{CT}_{\{t\}} \prod_{n_0+1\le i<j\le n} \bigg(1-q^k \frac{t_i}{t_j} \bigg) \bigg(1-q^{k+1} \frac{t_j}{t_i} \bigg)
\prod_{i=1}^n (t_i)_a \bigg(\frac{q}{t_i} \bigg)_b
\prod_{1\le i< j\le n} \bigg( q\frac{t_j}{t_i} \bigg)_k \bigg( \frac{t_i}{t_j} \bigg)_k \nonumber
\\ \qquad
{}=M_{n_0}(a,b,k;q) \nonumber
\\ \qquad\hphantom{=}
{}\times\prod_{j=0}^{n_1-1}\frac{[(k\!+1)(j\!+1)]_q\Gamma_q((k\!+1)j\!+a\!+b\!+kn_0\!+1) \Gamma_q((k\!+1)(j\!+1)\!+kn_0)}
{ \Gamma_q(2+k)\Gamma_q((k+1)j+a+kn_0+1)\Gamma_q((k+1)j+b+kn_0+1)},
\label{eq1.3}
\end{gather}
 where $ n=n_0+n_1$ and $ \mbox{CT}_{\{t\}}$ denotes the constant term of the Laurent polynomial in
 $t=(t_1,\dots,t_n)$. When $n_0=0$ this is nothing but the $q$-Morris constant term identity~\cite{Mo} and
 is equivalent to the $q$-Selberg integral formula~\eqref{eq1.1}. Similarly the Baker--Forrester constant term
identity~\eqref{eq1.3} is equivalent to an evaluation of the $q$-integral
\begin{gather}
\int_{[0,1]^n} \prod_{i=1}^{n_0} t_i^{n_1-1} \prod_{n_0+1\le i<j\le n}\big(t_i - q^{-k}t_j\big)\big(t_ i- q^{k+1}t_j\big)
 \prod_{i=1}^n t_i^{\alpha-1}\frac{(qt_i)_{\infty}}{\big(q^{\beta}t_i\big)_{\infty}}\nonumber
 \\ \qquad
 {}\times \prod_{1\le i<j\le n}
t_i^{2k}\bigg(q^{1-k}\frac{t_j}{t_i}\bigg)_{2k}\, {\rm d}_qt_1\cdots {\rm d}_qt_n.\label{eq1.4}
 \end{gather}
 In \cite{Kan2} we sought to evaluate~\eqref{eq1.4} using an integration formula for Macdonald
polynomials \cite{Kan1} but only succeeded for some special cases
 (see also \cite{B,GLXZ,Ham} for partial results). Then we turned to the quadratic norm formula of
 Koornwinder polynomials and the results obtained are evaluations of similar but different integrals,
 see~\eqref{eq7.2} and~\eqref{eq7.4} in Theorem~\ref{theorem7.3}. The constant term conjecture~\eqref{eq1.3}
 has recently been proved by K\'{a}rolyi et al.~\cite{Kar} employing the combinatorial Nullstellensatz.
This establishes the evaluation formula of the $q$-integral~\eqref{eq1.4}, see~\eqref{eq7.6} in~Theo\-rem~\ref{theorem7.4} for the explicit form. Presently it is not clear that the conjecture can be
 proved by suitably adapting the method of this paper.

As stated above, the proof of Theorem~\ref{theorem7.3} relies on the quadratic norm formula of Koornwinder
polynomial \cite {Ma2, St2} (see~\eqref{eq3.4}, \eqref{eq3.5} for explicit form). Koornwinder polynomials are
 a~family of orthogonal Laurent polynomials in several
variables, introduced by Koornwinder~\cite{Ko} in the symmetric case, that generalize the Askey--Wilson
polynomials. We apply the quadratic norm formula for specific Koornwinder polynomial $E_{\rho_1}$ (Proposition~\ref{proposition3.1})
from which the quad\-ra\-tic norm of partially antisymmetrization $U_1^{-} E_{\rho_1}$ is derived, see
Theorem~\ref{theorem4.6}. Then we take a~sca\-ling limit of each side of the quadratic norm formula.
The limits are given in Theorems~\ref{theorem5.4} and~\ref{theorem6.2}, and equating theses limits yields Theorem~\ref{theorem7.3}. The procedure of taking a scaling limit of the quadratic norm formula was previously
considered by Stokman in \cite{St1}, where he treated the case of symmetric Koornwinder polynomials.
Our calculation might be considered as a specific but nonsymmetric case of his procedure, and is carried out
in a similar way, see Remark~\ref{remark6.3}.

 Our proof needs somewhat involved limit calculations and it would be interesting to see whether or not the method
 of K\'{a}rolyi et al.\ would give a combinatorial proof of our formula formulated as a constant term identity.

 The paper is organized as follows. After preparing necessary results on Koornwinder polynomials in Section~\ref{section2},
 we define a bilinear form and give the explicit formula of the quadratic norm of Koornwinder polynomials in Section~\ref{section3}.
 In Section~\ref{section4} we define partial antisymmetrization and calculate the quadratic norm of antisymmetrization of
 a specific Koornwinder polynomial. Then in Sections~\ref{section5} and~\ref{section6} we calculate a suitable scaling limit of both sides
 of the quadratic norm formula. Finally in Section~\ref{section7}, we rewrite the resulting limit formula in terms of a $q$-integral.
This gives generalizations of the $q$-Selberg integral formula, see Theorem~\ref{theorem7.3}. In particular,
as in the original Selberg case, generalizations of Mehta integral are evaluated as limit cases of our integral, see
 Corollaries~\ref{corollary7.7} and~\ref{corollary7.11}. Furthermore we shall show that
 these results are amenable to be stated uniformly
 in terms of degrees of finite reflection groups of
 classical type and their parabolic subgroups, see Proposition~\ref{proposition7.12}.

\section{Koornwinder polynomials}\label{section2}

In this section we summarize necessary results on Koornwinder polynomials. Our main references are \cite{St2, St3}.

We first recall the affine root system of type $C^{\vee} C_n$. We assume that $n\ge 2$ throughout the paper.
Let $V=(\mathbb R^n, (\,,)) $ be Euclidean $n$-space with orthonormal basis $ \{ \epsilon_n \}_{i=1} ^n $.
We denote $|v|=\sqrt{(v,v)}$ for the norm of $v\in V$.
Let $\widehat V$ be the vector space of affine linear functions from~$ V$ to $\mathbb{R}$. Each element of
$\widehat V$ may be written as
 \begin{equation*}
 f(w)=(v,w)+\lambda\delta(w),\qquad w\in V,
 \end{equation*}
for some $v\in V$ and $\lambda \in \mathbb R$, where $\delta$ is the function identically equal to one on $V$.
We extend the scalar product to a positive semi-definite bilinear form on $\widehat V$ by setting
 \begin{equation*}
 (v+\lambda \delta,w+ \mu \delta)= (v,w).
 \end{equation*}
For $f\in \widehat V \setminus \mathbb R \delta$ define the reflection $s_f$ by
 \begin{equation*}
 s_f(g)= g-\big(g,f^{\vee}\big)f, \qquad g\in \widehat V,
 \end{equation*}
where $ f^{\vee}=2f/|f|^2$ is the coroot of $f$. Note that $s_f(g) =g\circ \tilde{s}_f^{-1}$, where
$\tilde{s}_f\colon V\to V$ is the orthogonal reflection in the affine hyperplane $f^{-1}(0)$. For $v\in V$ we define
 the translation operator $\tau(v)\in \mathrm {GL}(\widehat V)$ by
 \begin{equation*}
 \tau(v)f= f+(f,v)\delta, \qquad f\in \widehat{V}.
 \end{equation*}
 Note that $\tau(v)f=f\circ \tilde{\tau}_v^{-1} $, where $\tilde{\tau}_v\colon V\to V $ is given by $ \tilde{\tau}_v(w)=w-v$.

We consider three kinds of root system. First the irreducible root system of type $C_n$ is defined by the following
finite subset $ \Sigma \subset V$,
\begin{equation*}
 \Sigma =\{\pm2\epsilon_i \mid 1\le i \le n\} \cup \{ \pm\epsilon_i\pm\epsilon_j \mid 1\le i<j \le n\},
 \end{equation*}
 where all the sign combinations occur.
 The Weyl group $W$ of the root system $C_n$, generated by the reflections
 $s_v, v \in \Sigma$, is the semi-direct produt $W=S_n\ltimes (\mathbb{Z}/2\mathbb{Z})^n$ given by permutations
 and sign changes of the fixed basis
 $ \{ \epsilon_n \}_{i=1} ^n $. We shall write $\Sigma =\Sigma_{l} \sqcup \Sigma_{m}$ for the decomposition of
 $\Sigma$ into $W$-orbits, $\Sigma_{l}$
 (resp.\ $\Sigma_{m}$) being the roots of squared length four (resp.\ two).
 Next let $R\subset \widehat{V} $ be the irreducible reduced affine root system of type $\widetilde C_n$:
\begin{equation*}
 R=\Sigma+\mathbb Z \delta.
 \end{equation*}
Then the irreducible non-reduced affine root system of type $C^{\vee} C_n$ is defined by
 \begin{equation*}
S= R\cup R^{\vee},
 \end{equation*}
 where $R^{\vee}$ be the dual root system of $R$. The affine Weyl groups of $R$ and $S$ are the same, generated
 by the reflections $s_f$, $f\in S$, and is denoted by $\mathcal W$.

Let $Q^{\vee}$ be the coroot lattice of $\Sigma$. $Q^{\vee}$ coincides with the weight lattice
$\Lambda$ of $\Sigma$ and
 \begin{equation*}
 Q^{\vee}=\Lambda=\bigoplus_{i=1}^{n}\mathbb Z \epsilon_i.
 \end{equation*}
One can readily show that
 \begin{equation*}
 \mathcal W =W\ltimes \tau\big(Q^{\vee}\big).
\end{equation*}

We fix a base $\{a_i \}_{i=1}^{n} $ of the root system $\Sigma$ by
\begin{equation*}
 a_i=\epsilon_i-\epsilon_{i+1}, \qquad i=1,\dots,n-1, \qquad a_n=2\epsilon_n.
 \end{equation*}
 Define
 \begin{equation*}
a_0=\delta-2\epsilon_1.
 \end{equation*}
 Then $\{a_i \}_{i=0}^{n} $ is a base of $R$
and the set $\big\{ a_0^{\vee}=a_0/2,a_ 1,\dots,a_{n-1}, a_n^{\vee}=a_n/2 \big\} $ is a base of $R^{\vee}$
and of $S$ as well. We denote by $\Sigma^+$ (resp.~$\Sigma^{-}$)
the positive (resp.~negative) roots in $\Sigma$ so that we~see
 \begin{gather*}
 \Sigma^+ =\{2\epsilon_i \mid 1\le i \le n\} \cup \{ \epsilon_i\pm\epsilon_j \mid 1\le i<j \le n\},
 \end{gather*}
and by $\Lambda^{+} =\oplus_{i=1}^n\mathbb Z_{\ge 0} \omega_i$
 the corresponding cone of dominant weights
where $\omega_i=\epsilon_1+\cdots+\epsilon_i$ $(i=1,\dots,n)$ are the fundamental weights of $\Lambda$.
We also write $ Q^{\vee,+} $ for the positive span of
the simple coroots $a_i^{\vee}$, $i=1,\dots,n$: $Q^{\vee,+} = \sum_{i=1}^{n}\mathbb Z_{\ge 0} a_i^{\vee}$.

Let $R^{+}$ (resp.~$R^{-}$) be the positive (resp.~negative) roots of $R$ with respect to the base above. Observe that
 \begin{equation*}
 R^{+} =\Sigma^{+} \cup \{f \in R \mid f(0) >0 \}.
 \end{equation*}

The affine Weyl group $\mathcal W$ is generated by $s_i=s_{a_i}$, $i=0,\dots,n $ and is known to be isomorphic to
the Coxeter group with
 generators $s_i$, $(i=0,\dots,n)$ satisfying $s_i^2=1$ and the braid relations
 $ s_is_{i+1}s_is_{i+1}=s_{i+1}s_is_{i+1}s_i$,
 $i=0, n-1$, $s_is_{i+1}s_i=s_{i+1}s_is_{i+1}$, $i=1,\dots,n-2$, $s_is_j= s_js_i$ for~$|i-j|\ge 2$.

 Let $\mathcal A$ be the group algebra of the weight lattice $\Lambda$ with the basis
 $x^{\lambda}$, $\lambda\in \Lambda$, $x^0=1$ the unit element.
 The group algebra $\mathcal A$ is isomorphic to the Laurent polynomials in the $n$ indeterminates
 $x_i=x^{\epsilon_i}$, $i=1,\dots,n$. We fix a generic parameter $q\in \mathbb{C} \setminus \{0\} $
 and let $q^{\frac{1}{2}}$ be a fixed square root of $q$. We define
 \begin{equation*}
 x^{\lambda+c\delta}= q^cx^{\lambda},\qquad
 \lambda \in \Lambda, \qquad
 c\in \dfrac{1}{2} \mathbb Z.
 \end{equation*}
 Then the map $w(x^{\lambda})=x^{w\lambda}$ for $w\in \mathcal W$ and $\lambda \in\Lambda$ extends
 by linearity to an action of~$ \mathcal W$ on~$\mathcal A$. In particular, we have
 \begin{gather*}
(s_0 f)(x)= f\big(qx_1^{-1},x_2,\dots,x_n\big),
\\
(s_i f)(x)= f(x_1,\dots,x_{i-1},x_{i+1},x_i,x_{i+2},\dots,x_n), \qquad i=1,\dots,n-1,
\\
(s_n f)(x)= f(x_1,x_2,\dots,x_{n-1},x_n^{-1}),
 \end{gather*}
where $f\in \mathcal A$ and $x=(x_1,\dots,x_n)$. Note that the elements
$\tau(\lambda)\in \mathcal W \, (\lambda\in \Lambda)$
act as $q$-difference operators:
$\tau(\lambda)(x^{\mu}) =q^{(\lambda,\mu)}x^{\mu}$ for all $\mu \in \Lambda$.

We next construct a five parameter deformation of the above action of $\mathcal W$ on $\mathcal A$.
Let $\mathbf t =(t_\beta)_{\beta \in S}$ be a $\mathcal W$-invariant map from $S$ to $ \mathbb{C} \setminus \{0\} $,
the multiplicity function of $S$. This function is determined by the five values
$t _{a_0^{\vee}}$, $t _{a_0}$, $t _{a_i}$ ($i \ne 0,n$), $t _{a_n}$, $t _{a_n^{\vee}}$ ,
which we shall write $t_0^{\vee }$, $t_0$, $t$, $t_n$, $t_n^{\vee}$ respectively for short. We extend this map to $\widehat V$
by setting $ t_f=1$ for $f\in\widehat V\setminus S$. For convenience,
we write $\mathbf k =(t_{\beta})_{\beta \in R}= (t_0,t,t_n) $
and $\mathbf k^{\vee} =(t_{\beta})_{\beta \in R^{\vee}}= \big(t_0^{\vee},t,t_n^{\vee}\big) $ for the restrictions of
$\mathbf t$ to $R$ and $R^{\vee}$ respectively. We shall need a dual multiplicity function $\tilde{\mathbf t}$
given by
 \begin{equation} \label{eq2.1} 
 \tilde{t}_0=t_n^{\vee},\qquad
 \tilde{t}_0^{\vee} =t_0^{\vee}, \qquad
 \tilde{t}= t, \qquad
 \tilde{t}_n^{\vee}=t_0, \qquad
 \tilde{t}_n=t_n,
 \end{equation}
 which we shall write as $ t_{i}$ \big(resp.~$\tilde{t}_i$\big) for $ t_{a_i}$ \big(resp.~$\tilde{t}_{a_i} $\big), $i=0,\dots,n$.

\begin{Definition} 
 The affine Hecke algebra $H=H(R;\mathbf k)$ of type $\widetilde C _n$ is the unital, associative algebra with generators $T_0,\dots,T_n$ and relations
 \begin{equation*}
 (T_i-t_i)\big(T_i+t_i^{-1}\big)=0, \qquad i=0,\dots,n,
 \end{equation*}
 and the braid relations
 \begin{gather*}
 T_iT_{i+1}T_iT_{i+1}=T_{i+1}T_iT_{i+1}T_i, \qquad i=0, n-1,
 \\
 T_iT_{i+1}T_i =T_{i+1}T_iT_{i+1}, \qquad i=1,\dots,n-2,
 \\
 T_iT_j=T_jT_i, \qquad |i-j|\ge 2.
 \end{gather*}
 \end{Definition}

Similarly one has the affine Hecke algebra $H\big(R^{\vee};\mathbf k^{\vee}\big)$ in which the parameter $t_i$ is replaced
by $t_i^{\vee}$, $i=0,\dots,n$.
 For a reduced expression $w=s_{i_1}\cdots s_{i_r} $ of $w\in \mathcal W$ we set $T_w=T_{i_1}\cdots T_{i_r}$.

 The Cherednik--Dunkl type $Y$-operator is defined by
 \begin{equation*}
 Y_i= T_i\cdots T_{n-1}T_n\cdots T_0 T_1^{-1}T_2^{-1}\cdots T_{i-1}^{-1},\qquad
 i=1,\dots,n
 \end{equation*}
which is an analogue of the reduced expression of the translation $\tau(\epsilon_1)$:
 \begin{equation*}
 \tau(\epsilon_i) = s_i \cdots s_{n-1} s_n\cdots s_0 s_1\cdots s_{i-1}.
 \end{equation*}
 Lusztig \cite{L} has shown that the ${Y_i} $ commute pairwise and generate a commutative
 subalgebra~$\mathcal A_Y$ of $H$ which is isomorphic to $\mathcal A$.

Let us define rational functions $v_{\beta}(x) =v_{\beta}(x;\mathbf t;q) $ for $\beta \in R$ by
 \begin{equation} \label{eq2.2} 
 v_{\beta}(x;\mathbf t;q)=\frac{\big(1-t_{\beta}t_{\beta/2}x^{\beta/2}\big) \big(1+t_{\beta}t_{\beta/2}^{-1}x^{\beta/2}\big)}{1-x^{\beta}}.
 \end{equation}
Noumi \cite{N} (cf.\ \cite[Theorem 4.14]{St2}) has shown that the following assignment extends to a~representation
$\pi_{\mathbf { t,q}}\colon H(R;\mathbf k)\to \text{\rm{End}} _{\mathbb C}(\mathcal A$).
 \begin{equation*}
 T_i \mapsto t_i+ t_i^{-1}v_{a_i}(x;\mathbf t;q)(s_i-1) \in \text{\rm{End}} _{\mathbb C}(\mathcal A), \qquad
 i=0,\dots,n.
 \end{equation*}
 We write $Y_i$ for the image of $Y_i\in H(R;\mathbf k)$ under the representation $\pi_{\mathbf {t,q}}$.

Before proceeding to the definition of Koornwinder polynomials we need to define two partial orders
on the weight lattice $\Lambda$.
For $\lambda \in \Lambda$, let $\lambda^{+} \in \Lambda^{+} $
be the unique dominant weight in the orbit $W\lambda$.
\begin{Definition} 
Let $ \lambda, \mu \in \Lambda$.
\begin{enumerate}\itemsep=0pt
\item We write $\lambda \leq \mu$ if $\mu-\lambda \in Q^{\vee,+} $ (and $\lambda < \mu$ if
$\lambda \leq \mu$ and $\lambda \ne \mu$).

\item We write $ \lambda \preceq \mu $ if $ \lambda^{+}< \mu^{+} $, or if $ \lambda^{+}= \mu^{+} $ and
$ \lambda\leq \mu$ (and $ \lambda \prec \mu $ if $ \lambda \preceq \mu $ and $\lambda \ne \mu $).
\end{enumerate}
\end{Definition}
Let $\epsilon\colon \mathbb Z \to \{ \pm 1\}$ be the function which maps a nonnegative integer to 1 and
a negative integer to $ -1$.
Let $\Sigma_m^{+}$ (resp.~$\Sigma_{l}^{+} $) be the positive roots in $\Sigma$ of squared length 2 (resp.~4),
and~put
 \begin{equation*}
 \rho^{(m)}=2\sum_{i=1}^{n}(n-i)\epsilon_{i}, \qquad
 \rho^{(l)}= \sum_{i=1}^{n} \epsilon_i
 \end{equation*}
for the sum of coroots $\alpha^{\vee}$ with $\alpha \in \Sigma _m^{+}$ and $\alpha \in \Sigma _l^{+}$ respectively.
For an element of $\lambda \in \Lambda$ we~set
 \begin{equation*}
 \rho^{(m)}(\lambda)= \sum_{\alpha\in\Sigma_m^{+}} \epsilon((\lambda,\alpha)) \alpha^{\vee},
 \qquad
 \rho^{(l)}(\lambda)=\sum_{\alpha\in\Sigma_l^{+}} \epsilon((\lambda,\alpha)) \alpha^{\vee}.
 \end{equation*}
 Then it holds that (see \cite[proof of Lemma 5.6]{St2})
 \begin{equation*}
 \rho^{(m)}(\lambda)=w_{\lambda} \rho^{(m)}, \qquad
 \rho^{(l)}(\lambda) =w_{\lambda} \rho^{(l)}, \qquad \lambda \in \Lambda,
 \end{equation*}
 where $w_{\lambda} \in W$ is the unique element of minimal length such that
 $w_{\lambda} \lambda^{+} = \lambda $.
Let $\gamma_{\lambda}=\gamma_{\lambda}(\mathbf k,q) \allowbreak\in \mathbb C^n $ be the vector with $ i$th coordinate
given by
 \begin{equation*}
 \gamma_{\lambda,i}=(t_0t_n)^{(\rho^{(l)}(\lambda),\epsilon_i)} t^{(\rho^{(m)}(\lambda),\epsilon_i)}
 q^{(\lambda,\epsilon_i)},\qquad
 i=1,\dots,n.
 \end{equation*}
In particular for a dominant weight $\lambda \in \Lambda^{+}$ one sees
 \begin{equation*}
 \gamma_{\lambda}= \big(t_0t_nt^{2(n-1)}q^{\lambda_1},t_0t_nt^{2(n-2)}q^{\lambda_2},\dots,t_0t_nq^{\lambda_n}\big).
 \end{equation*}

 Now since the Cherednik--Dunkl $Y$-operators $Y_i$, $i=1,\dots,n$ act triangularly on $\mathcal{A}$
 with respect to the partial order $\preceq$ \cite[Proposition~4.5]{St2} and $\gamma_{\lambda,i}$ are
 mutually distinct for gene\-ric~$q$,~$\mathbf k $,
 we have \cite[Theorem 4.8]{St2}

 \begin{Theorem} 
 There exists a unique basis $ \{E_{\lambda} \}_{\lambda \in \Lambda} $ of $\mathcal A$ such that
 \begin{enumerate}\itemsep=0pt
\item[$1)$] $ E_{\lambda}(x)=x^{\lambda}+ \sum_{\mu\prec \lambda} c_{\lambda,\mu}x^{\mu} $
 {\it for certain constants} $c_{\lambda,\mu}$,

 \item[$2)$] $ Y_i E_{\lambda} = \gamma_{\lambda,i} E_{\lambda} $.
 \end{enumerate}
 \end{Theorem}

\begin{Definition}
 The Laurent polynomial $E_{\lambda}(x)=E_{\lambda}(x; \mathbf t;q)$ is called the
 monic, nonsymmetric Koornwinder polynomial of degree $\lambda$.
 \end{Definition}

 \section{Quadratic norm formula}\label{section3}

 In this section we define a bilinear form on $\mathcal A$ and give an explicit form for the quadratic norm of the
 Koornwinder polynomial $E_{\lambda}$. We assume that $0<q, t <1$.

 We put
 \begin{equation} \label{eq3.1} 
 a_1=t_0t_0^{\vee}q^{1/2},\qquad
 a_2= -t_0\big(t_0^{\vee}\big)^{-1}q^{1/2},\qquad
 a_3=t_nt_n^{\vee}, \qquad
 a_4=-t_n\big(t_n^{\vee}\big)^{-1},
 \end{equation}
 and assume also that $ |a_r|<1$, $r=1,2,3,4$. Note that, once the quadratic norm formula is established,
 we shall take the limit $a_1\to\infty$ or $a_3\to\infty$ in this formula. For simplicity we will write
 \begin{equation*}
A=a_1a_2a_3a_4 =qt_0^2t_n^2.
 \end{equation*}

 Define $\Delta(x)=\Delta(x;\mathbf t;q)$ and $\Delta_{+}(x)=\Delta_{+}(x;\mathbf t;q)$ by
\begin{equation*}
\Delta(x;\mathbf t;q)=\prod_{\beta\in R^{+}}\frac{1}{v_{\beta}(x;\mathbf t;q)}
 \end{equation*}
 and
 \begin{align*}
 \Delta_{+}(x;\mathbf t;q) &= \prod_{\beta\in R,\,\beta(0)\geq 0}\frac{1}{v_{\beta}(x;\mathbf t;q)}
 \\
 & = \prod_{1\leq i<j \leq n}
 \frac{\big(x_ix_j,x_ix_j^{-1},x_i^{-1}x_j,x_i^{-1}x_j^{-1}\big)_{\infty}}
 {\big(t^2x_ix_j,t^2x_ix_j^{-1},t^2x_i^{-1}x_j,t^2x_i^{-1}x_j^{-1}\big)_{\infty}}
 \prod_{i=1}^n
 \frac{\big(x_i^2,x_i^{-2}\big)_{\infty}}{\prod_{r=1}^4\big(a_rx_i,a_rx_i^{-1}\big)_{\infty}}.
 \end{align*}
 Then
 \begin{equation*}
 \Delta(x;\mathbf t;q)=\mathcal C(x;\mathbf t;q)\Delta_{+}(x;\mathbf t;q),
 \end{equation*}
 where $\mathcal C(x;\mathbf t;q)$ is given by
 \begin{equation*}
 \mathcal C(x;\mathbf t;q) = \prod_{\alpha\in \Sigma^{-}} v_{\alpha}(x;\mathbf t;q)
 = \prod_{1\leq i<j \leq n} \frac{\big(x_ix_j-t^2\big)\big(x_ix_j^{-1}-t^2\big)}{(x_ix_j-1)\big(x_ix_j^{-1}-1\big)}
 \prod_{i=1}^n
 \frac{(x_i-a_3)(x_i-a_4)}{x_i^2-1}.
 \end{equation*}
 Let $\dagger$ be an involution on $\mathcal A$ which maps $x_i \mapsto x_i^{-1}$ and
 $(\mathbf t,q)\mapsto \big({\mathbf t}^{-1},q^{-1}\big)$, where $\mathbf t^{-1}=\big(t_\beta^{-1}\big)_{\beta \in S}$.
We now define a bilinear form $\langle\:\, , \: \rangle =\langle\:\, , \:\rangle_{\mathbf t, q}$ on $\mathcal A$ by
 \begin{equation*}
\langle f,g \rangle = \frac{1}{(2\pi {\rm i})^n}\int_{C^n} f(x) g^{\dagger}(x)\Delta(x) \frac{{\rm d}x}{x},
 \end{equation*}
 where $\frac{{\rm d}x}{x}=\frac{{\rm d}x_1}{x_1}\dots \frac{{\rm d}x_n}{x_n}$ and $C$ is the positively oriented unit circle
 around the origin. Note that when $\Delta(x) \in \mathcal A$,
 $\langle f,g\rangle $ equals the constant term of the Laurent polynomial $f(x)g^{\dagger}(x)\Delta(x)$.

 The following property due to Sahi \cite[Theorem 3.1]{Sa}, (cf.\ \cite[Proposition 8.3]{St2}) is
 fundamental to our calculation:
 For $f,g\in \mathcal A$, we have
 \begin{equation} \label{eq3.2} 
 \langle T_i f,g \rangle
 = \big\langle f,T_i^{-1} g \big\rangle, \qquad i=0,\dots, n.
 \end{equation}

We now state the quadratic norm formula for a dominant weight $\lambda \in \Lambda^{+}$
 \cite[Theorems 8.10 and~9.3]{St2}.
\begin{gather} 
 \langle E_{\lambda}, E_{\lambda}\rangle =
 \prod_{i=1}^{n}
 \frac{\big(A q^{2\lambda_i}t^{4(n-i)}\big)_{\infty}^2\big(1-a_3a_4q^{\lambda_i}t^{2(n-i)}\big)}
 {\big(q^{\lambda_i+1}t^{2(n-i)},Aq^{\lambda_i}t^{2(n-i)}\big)_{\infty}
 \prod_{1\le r<s\le 4}\big(a_ra_s q^{\lambda_i}t^{2(n-i)}\big)_{\infty}} \nonumber
 \\ \hphantom{\langle E_{\lambda}, E_{\lambda}\rangle = }
{}\times \prod_{1 \le i<j\le n}
\frac{\big(q^{\lambda_i-\lambda_j+1}t^{2(j-i)}\big)_{\infty}^2} {\big(q^{\lambda_i-\lambda_j+1}t^{2(j-i+1)},
q^{\lambda_i-\lambda_j+1}t^{2(j-i-1)}\big)_{\infty}} \nonumber
\\ \hphantom{\langle E_{\lambda}, E_{\lambda}\rangle = \prod_{1 \le i<j\le n}}
{}\times \frac{\big(Aq^{\lambda_i+\lambda_j}t^{2(2n-i-j)}\big)_{\infty}^2}
{\big(Aq^{\lambda_i+\lambda_j}t^{2(2n-i-j+1)},Aq^{\lambda_i+\lambda_j}t^{2(2n-i-j-1)}\big)_{\infty}}.
\label{eq3.3}
\end{gather}
For a general weight $\lambda \in \Lambda$, one has \cite[Lemma 9.1 and Theorem 9.3]{St2}
 \begin{equation} \label{eq3.4} 
 \frac{\langle E_{\lambda}, E_{\lambda} \rangle }{\langle E_{\lambda^+}, E_{\lambda^+} \rangle }=
 \prod_{\alpha\in\Sigma^+ \cap w_{\lambda}^{-1} \Sigma^{-}}
 \frac{1}{\widetilde v_{\alpha}\big(\gamma_{\lambda^+}^{-1}\big)
 \widetilde v_{-\alpha}\big(\gamma_{\lambda^+}^{-1}\big)^{\dagger}}.
 \end{equation}

Let $n=n_0+n_1$ with $n_1\ge 2$. We define a subroot system $ \Sigma_1$ in $\Sigma$ by
 \begin{equation*}
 \Sigma_1=\{\pm2\epsilon_i \mid n_0+1\le i \le n \}\cup \{ \pm\epsilon_i\pm\epsilon_j \mid n_0+1\le i<j \le n\},
 \end{equation*}
 with a base $\{a_i \}_{i=n_0+1}^{n} $, so that positive roots $\Sigma_1^+$
 \big(negative roots: $\Sigma_{1}^{-}=-\Sigma_1^{+}$\big)
 are given by
 \begin{equation*}
 \Sigma_1^+ =\{2\epsilon_i \mid n_0+1\le i \le n \}\cup \{ \epsilon_i\pm\epsilon_j \mid n_0+1\le i<j \le n\}.
 \end{equation*}
The subgroup of $W$ generated by simple reflections $ \{ s_{n_0+1},\dots,s_n\}$ will be denoted by $W_1$.

 Let $\rho_1$ be a weight in $\Lambda$ defined by
\begin{equation*}
 \rho_1=\frac{1}{2} \sum_{\alpha\in \Sigma_{1}^{+}} \alpha=\big(0^{n_0},n_1,n_1-1,\dots,1\big).
\end{equation*}

We write down explicitly the norm $\langle E_{ \rho_1}, E_{ \rho_1}\rangle $.
Since $ \rho_1^+= (n_1,\dots,1,0^{n_0})$, the formula for the
the norm $ \bigl< E_{\rho_1^+}, E_{ \rho_1^+}\bigr>$ is immediate from~\eqref{eq3.3}.
Moreover from~\eqref{eq3.4} we find
\begin{equation} \label{eq3.5} 
 \frac{\langle E_{ \rho_1}, E_{ \rho_1} \rangle }{\bigl< E_{ \rho_1^+}, E_{ \rho_1^+} \bigr>}=
 \prod_{i=1}^{n_1}
\prod_{j=n_1+1}^{n} \frac{\big(1-q^{n_1-i+1} t^{2(j-i)}\big)^2}{ \big(1-q^{n_1-i+1}t^{2(j-i+1)}\big)\big(1-q^{n_1-i+1}t^{2(j-i-1)}\big)}.
\end{equation}
In fact, to see this, it is sufficient to note that the element $w_{ \rho_1} \in W $ is given by the permutation
\begin{equation*}
\genfrac{(}{)}{0mm}{0}{1,\dots,n_0,n_0 +1,\dots, n}{n_1+1,\dots,n,1,\dots,n_1}
\end{equation*}
so that
\begin{equation*}
\Sigma^{+}\cap w_{ \rho_1}^{-1} \Sigma^{-}=\{\epsilon_i-\epsilon_j\colon 1\le i\le n_1, \, n_1+1\le j \le n \},
\end{equation*}
and that
\begin{equation*}
\gamma_{ \rho_1^{+}}= \big(t_0t_nt^{2(n-1)}q^{n_1},\dots,t_0t_nt^{2(n-n_1)}q,t_0t_nt^{2(n-n_1-1)},\dots,
t_0t_n\big).
\end{equation*}
Combining~\eqref{eq3.3} and~\eqref{eq3.5}, we arrive at the explicit formula of the norm
$ \langle E_{ \rho_1}, E_{ \rho_1}\rangle $.
\begin{Proposition} \label{proposition3.1} 
For the weight $ \rho_1=(0^{n_0}, n_1,n_1-1,\dots,1)$, we have
 \begin{gather} 
 \langle E_{ \rho_1}, E_{ \rho_1}\rangle = \prod_{i=1}^{n_1}\frac{\big(Aq^{2(n_1-i+1)}t^{4(n-i)}\big)_{\infty}^2 \big(1-a_3a_4q^{n_1-i+1}t^{2(n-i)}\big)}
 {\big(q^{n_1-i+2}t^{2(n-i)},Aq^{n_1-i+1}t^{2(n-i)}\big)_{\infty} \prod_{1\le r<s\le 4} \big(a_ra_sq^{n_1-i+1}t^{2(n-i)}\big)_{\infty}} \nonumber
\\ \hphantom{\langle E_{ \rho_1}, E_{ \rho_1}\rangle =}
 {}\times \prod_{i=n_1+1}^n\frac{\big(At^{4(n-i)}\big)_{\infty}^2
 \big(1-a_3a_4t^{2(n-i)}\big)}{\big(qt^{2(n-i)},a_0t^{2(n-i)}\big)_{\infty}
 \prod_{1\le r<s\le 4}\big(a_ra_st^{2(n-i)}\big)_{\infty}} \nonumber
 \\ \hphantom{\langle E_{ \rho_1}, E_{ \rho_1}\rangle =}
{}\times \prod_{1\le i<j\le n_1} \frac{\big(q^{j-i+1}t^{2(j-i)}\big)_{\infty}^2}
 {\big(q^{j-i+1}t^{2(j-i+1)},q^{j-i+1}t^{2(j-i-1)}\big)_{\infty}} \nonumber
 \\ \hphantom{\langle E_{ \rho_1}, E_{ \rho_1}\rangle =\times \prod_{1\le i<j\le n_1}}
{}\times \frac{\big(Aq^{2n_1-i-j+2}t^{2(2n-i-j)}\big)_{\infty}^2}{\big(Aq^{2n_1-i-j+2}t^{2(2n-i-j+1)},
Aq^{2n_1-i-j+2}t^{2(2n-i-j-1)}\big)_{\infty}} \nonumber
\\ \hphantom{\langle E_{ \rho_1}, E_{ \rho_1}\rangle =}
{}\times\prod_{i=1}^{n_1} \prod_{j=n_1+1}^{n}
 \frac{\big(q^{n_1-i+1}t^{2(j-i)}\big)_{\infty}^2}{\big(q^{n_1-i+1}t^{2(j-i+1)},
 q^{n_1-i+1}t^{2(j-i-1)}\big)_{\infty}} \nonumber
 \\ \hphantom{\langle E_{ \rho_1}, E_{ \rho_1}\rangle =\times\prod_{i=1}^{n_1}\prod_{j=n_1+1}^{n}}
{}\times \frac{\big(Aq^{n_1-i+1}t^{2(2n-i-j)}\big)_{\infty}^2}{\big(Aq^{n_1-i+1}t^{2(2n-i-j+1)},
 Aq^{n_1-i+1}t^{2(2n-i-j-1)}\big)_{\infty}} \nonumber
 \\ \hphantom{\langle E_{ \rho_1}, E_{ \rho_1}\rangle =}
{}\times\prod_{n_1+1\le i<j\le n}
 \frac{\big(qt^{2(j-i)},At^{2(2n-i-j)}\big)_{\infty}^2}{\big(qt^{2(j-i+1)},qt^{2(j-i-1)}, At^{2(2n-i-j+1)},At^{2(2n-i-j-1)}\big)_{\infty}}.\label{eq3.6}
 \end{gather}
 \end{Proposition}

 \section{Partial antisymmetrization} \label{section4}

 In this section we define partial (anti)symmetrization and calculate the norm of partial antisymmetrization of the
 Koornwinder polynomial $E_{ \rho_1} $.

 Let $l\colon W_1\to \mathbb Z_{\ge 0}$ be the length function on $W_1$ defined by
 \begin{equation*}
 l(w)=\big\vert \Sigma_1^{+} \cap w^{-1} \Sigma_1^{-} \big\vert.
 \end{equation*}
The length function $l(w)$ is equal to the number $\mathit{r}$ of the reduced expression $ w=s_{i_1}\cdots s_{i_r}$
(cf.\ \cite{Hu}).
 Let $t_w$ be a function on $W_1$ defined
by $t_w=t_{i_1}\cdots t_{i_r} $ for a reduced expression $ w=s_{i_1}\cdots s_{i_r}$.
We define partial symmetrization $U^{+}_1$ and partial antisymmetrization $U^{-}_1$ as follows:
 \begin{equation} \label{eq4.1} 
 U^{\pm}_1=\frac{1}{\sum_{w\in W_1} t_w^ {\pm 2} }\sum_{w\in W_1} (\pm1)^{l(w)} t_w^{\pm1} T_w.
\end{equation}
 Then the following fundamental properties hold.

 \begin{Proposition}
 \begin{gather} 
 \big(T_i\mp t_i^{\pm 1}\big)\big( U^{\pm}_1\big) = 0 \qquad for \quad i=n_0+1,\dots,n, \label{eq4.2}
 \\
 {U{_1}^{\pm}} ^2 = U^{\pm}_1, \label{eq4.3}
 \\
 \langle U^{\pm}_1f, g \rangle = \langle f, U^{\pm}_1g \rangle . \label{eq4.4}
\end{gather}
 \end{Proposition}

\begin{proof}
The proof is carried out in exactly the same way as in \cite[Lemma 8.9]{St2}.
For~\eqref{eq4.2}, fix $i \in \{n_0+1,\dots,n\}$ and decompose $W_1=W_{1,i}^{+}\sqcup W_{1,i}^{-} $, where
$W_{1,i}^{\pm}=\{w\in W_1\colon l(s_iw)=l(w)\pm 1 \}$. Note that $W_{1,i}^{-}=s_i W_{1,i}^{+}$ and that
$t_{s_iw}=t_it_w, T_{s_iw}=T_iT_w$ for $w\in W_1^{+}$. Then
\begin{align*}
T_i\bigg(\sum_{w\in W_1} (\pm1)^{l(w)} t_w^{\pm1} T_w\bigg) &=
\sum_{w\in W_{1,i}^{+}}\bigl((\pm 1)^{l(w)} t_w^{\pm 1}T_iT_w+ (\pm 1)^{l(w)+1}
t_{s_iw}^{\pm 1}T_iT_{s_iw}\bigr)
\\
 & = \sum_{w\in W_{1,i}^{+}}(\pm 1)^{l(w)}t_w^{\pm 1}\big( T_i \pm t_{i}^{\pm 1}T_i^2\big)T_w
 \\
& = \sum_{w\in W_{1,i}^{+}}(\pm 1)^{l(w)}t_w^{\pm 1}\big( \pm t_i^{\pm1}+t_i^{\pm 2} T_i \big)T_w
 \\
& = \pm t_i^{\pm 1} \sum_{w\in W_{1}} (\pm1)^{l(w)} t_w^{\pm1} T_w,
\end{align*}
where the quadratic relations $ (T_i-t_i)\big(T_i+t_i^{-1}\big)=0 $ are applied to the third equality. Equation~\eqref{eq4.3}
follows from~\eqref{eq4.2} immediately.
For the proof of~\eqref{eq4.4}, in view of~\eqref{eq3.2}, it suffices to show that
\begin{equation*}
U_1^{\pm}= \frac{1}{\sum_{w\in W_1} t_w^ {\mp 2} }\sum_{w\in W_1} (\pm1)^{l(w)} t_w^{\mp1} T_w^{-1}.
\end{equation*}
Let $w_1$ be the longest element in $W_1$. Then $ l(w_1w)=l(w_1)-l(w)$ for $w\in W_1$, so that
$t_{w_1 w}=t_{w_1}t_w^{-1}$ and $T_{w_1 w}=T_{w_1}T_w^{-1}$ for $w\in W_1^{+}$. We change the summation
over $w\in W_1$ in~\eqref{eq4.1} to a summation over $w_1w$, $w\in W_1$. Observe that
\begin{gather*}
\sum_{w\in W_1} t_w^ {\pm 2} =t_{w_1}^{\pm 2} \sum_{w\in W_1} t_w^ {\mp 2}, \qquad
 \sum_{w\in W_1} (\pm1)^{l(w)} t_w^{\pm1} T_w
=(\pm 1)^{l(w_1)}t_{w_1}^{\pm 1}T_{w_1}\sum_{w\in W_1} (\pm 1)^{l(w)}t_{w}^{\mp}T_w^{-1}.
\end{gather*}
Substituting these equations into~\eqref{eq4.1}, and using the relation
\begin{equation*}
T_i(\sum_{w\in W_1}(\pm1)^{l(w)}t_{w}^{\mp}T_w^{-1})=\pm t_i^{\pm} \sum_{w\in W_1}(\pm1)^{l(w)}t_{w}^{\mp}T_w^{-1},
\end{equation*}
whose proof is similar to that of~\eqref{eq4.2}, we have
\begin{align*}
U_1^{\pm} & =(\pm 1)^{l(w_1)}t_{w_1}^{\mp 1}T_{w_1}\frac{1}{\sum_{w\in W_1} t_w^ {\mp 2} }\sum_{w\in W_1}
 (\pm1)^{l(w)}t_{w}^{\mp}T_w^{-1}
 \\
 & = \frac{1}{\sum_{w\in W_1} t_w^ {\mp 2} }\sum_{w\in W_1}
 (\pm1)^{l(w)}t_{w}^{\mp}T_w^{-1},
 \end{align*}
as desired.
 \end{proof}

We need to know the quadratic norm of the partial antisymmetrization $U_1^{-} E_{\rho_1}(x)$.
For this purpose it suffices to calculate
the expansion coefficient of $ E_{\rho_1}(x)$ in $U_1^{-} E_{\rho_1}(x)$.
The calculation is based on the following fundamental formula \cite[Proposition 6.1]{St2}:
\begin{equation} \label{eq4.5} 
T_i E_{\lambda}=\xi_i(\gamma_{\lambda})E_{\lambda}+\eta_i(\gamma_{\lambda})E_{s_i\lambda}
\end{equation}
with
\begin{equation*}
\xi_i(x)=\tilde{t}_i-\tilde{t}_i^{-1}\tilde{v}_{-a_i}(x)= \frac{\big(\tilde{t}_{a_i}^{-1}-\tilde{t}_{a_i}\big)x^{a_i}+\big(\tilde{t}_{a_i/2}^{-1}
-\tilde{t}_{a_i/2}\big)x^{a_i/2}}{1-x^{a_i}}
\end{equation*}
and
\begin{equation} \label{eq4.6} 
\eta_i(\gamma_{\lambda})=
\begin{cases}
\tilde{t}_i & \text{if} \quad \langle \lambda,a_i\rangle <0,
\\
\tilde{t}_i^{-3}\tilde{v}_{a_i}(\gamma_{\lambda})\tilde{v}_{-a_i}(\gamma_{\lambda})
& \text{if} \quad \langle \lambda,a_i\rangle \ge 0.
\end{cases}
\end{equation}

\begin{Proposition}
We have
\begin{equation*}
 U_1^{-} E_{\rho_1}(x)= \sum_{\mu\in W_1\rho_1} c_{\rho_1 \mu}E_{\mu}(x),
 \end{equation*}
 where
 \begin{gather}
 c_{\rho_1 \rho_1} = \frac{1}{\sum_{w\in W_1} \tilde t_w^{-2}}
 \prod_{1\le i< j \le n_1} \frac{\big(1-q^{j-i}t^{2(j-i-1)}\big)
 \big(1-Aq^{2n_1-i-j+1}t^{2(2n-i-j-1)} \big)}
 {\big(1-q^{j-i}t^{2(j-i)}\big)\big(1-Aq^{2n_1-i-j+1}t^{2(2n-i-j)} \big)} \nonumber
 \\ \hphantom{c_{\rho_1 \rho_1} = }
{}\times\prod_{i=1}^{n_1} \frac{\big(1-q^{n_1-i+1}t^{2(n-i)}\big)\big(1-a_1a_2q^{n_1-i}t^{2(n-i)}\big)} {1-a_0q^{2(n_1-i+1)-1}t^{4(n-i)}}.
\label{eq4.7}
 \end{gather}
\end{Proposition}
\begin{proof}
Put
\begin{equation*}
\rho_1^{R}=w_1 \rho_1= \big(0^{n_0}, -n_1,\dots,-1\big).
\end{equation*}
We will prove that
\begin{equation} \label{eq4.8} 
c_{\rho_1\rho_1^{R}} = \frac{1}{\sum_{w\in W_1} \tilde t_w^{-2}}(-1)^{l(w_1)}\tilde{t}^{-4}_{w_1}
\prod_{\alpha\in\Sigma_1^{+}}
\tilde{v}_{\alpha}(\gamma_{\rho_1})\tilde{v}_{-\alpha}(\gamma_{\rho_1}),
\end{equation}
and
\begin{equation} \label{eq4.9} 
\frac{c_{\rho_1\rho_1^{R}} }{ c_{\rho_1 \rho_1}}=(-1)^{l(w_1)}\tilde{t}^{-2}_{w_1}\prod_{\alpha\in\Sigma_1^{+}}
\tilde{v}_{\alpha}(\gamma_{\rho_1}),
\end{equation}
and hence
\begin{equation*}
c_{\rho_1\rho_1}=\frac{1}{\sum_{w\in W_1} \tilde t_w^{-2}}\tilde{t}^{-2}_{w_1}\prod_{\alpha\in\Sigma_1^{+}}
\tilde{v}_{-\alpha}(\gamma_{\rho_1}).
\end{equation*}

Let $w_1=s_{i_1}\cdots s_{i_r} $ ($r=n_1^2$, actually) be a reduced expression.
It follows from~\eqref{eq4.5} that the coefficient of
$E_{\rho_1^{R}}$ in the expansion of $T_{w_1}E_{\rho_1}$ is given by
$\prod_{j=1}^{r}\eta_ {i_j}(\gamma_{s_{i_{j+1}}\cdots s_{i_{r}} \rho_1}) $.
But since the product $s_{i_1}\cdots s_{i_r} $ is reduced, we see that
$(s_{i_{j+1}}\cdots s_{i_{r}})^{-1} a_{i_j}$ is positive\,(cf.\ \cite[Lemma 1.6]{Hu}),
so that $ \langle a_{i_j}, s_{i_{j+1}}\cdots s_{i_{r}}\rho_1 \rangle >0$. It also holds, by
virtue of \cite[Lemma 4.6]{St2}, that
\begin{equation} \label{eq4.10} 
\gamma_{s_{i_{j+1}}\cdots s_{i_{r}}\rho_1}^{a_{i_j}}=\gamma_{\rho_1}^{s_{i_r}\cdots s_{i_{j+1}}a_{i_j}}.
\end{equation}
Hence~\eqref{eq4.8} follows from the fact that
\begin{equation} \label{eq4.11} 
\{s_{i_r}\cdots s_{i_{2}}a_{i_1},\dots,a_{i_r} \}=\Sigma_1^{+}.
\end{equation}

For~\eqref{eq4.9}, we note that~\eqref{eq4.2} implies the following equality.
\begin{equation*}
(T_i+t_i^{-1})U_1^{-}E_{\rho_1}=0.
\end{equation*}
This is equivalent to the fact that the coefficients $c_{\rho_1 \mu}$ satisfy the recurrence equations
\begin{equation*}
c_{\rho_1 s_i\mu}=\bigg({-}\frac{\xi_i^{-}(\gamma_{\mu})}{\eta_i(\gamma_{s_i \mu})} \bigg)c_{\rho_1 \mu},
\end{equation*}
where
\begin{equation*}
\xi_i^{-}(x)= \xi_i(x)+t_i^{-1}=\tilde{t}_{i}^{-1}\tilde{v}_{a_i}(x).
\end{equation*}
provided $\langle \mu,a_i \rangle \ne 0$. Since $\langle a_{i_j},s_{i_j}\cdots s_{i_r}\rho_1 \rangle
=- \langle a_{i_j},s_{i_{j+1}}
\cdots s_{i_r}\rho_1 \rangle <0 $, it follows from~\eqref{eq4.6} that
\begin{equation*}
 \eta_i(\gamma_{s_{i_j}\cdots s_{i_r}\rho_1} ) =\tilde{t}_i.
\end{equation*}
Now~\eqref{eq4.9} is immediate from~\eqref{eq4.10} and~\eqref{eq4.11}.

Finally to see the explicit formula~\eqref{eq4.7}, we note that
\begin{equation*}
 f(\gamma_{\lambda})=w_{\lambda}^{-1}f(\gamma_{\lambda^{+}}), \qquad f\in \mathcal A,
 \qquad \lambda \in \Lambda,
\end{equation*}
again by \cite[Lemma 4.6]{St2}. This implies
\begin{equation*}
\tilde v _{-\alpha}(\gamma_{\rho_1})=\tilde v _{-w_{\rho_1}^{-1}\alpha}(\gamma_{\rho_1^{+}}),
\end{equation*}
where $w_{\rho_1}^{-1}$ is given by
\begin{equation*}
w_{\rho_1}^{-1}=\left ( 1,\dots,n_0,n_0+1,\dots,n \atop n_1+1,\dots,n,1,\dots,n_1 \right)\!.
\end{equation*}
Hence
\begin{equation*}
w_{\rho_1}^{-1} \Sigma_1^{+} =\{2\epsilon_i\mid 1\le i\le n_1 \}\cup \{ \epsilon_i\pm\epsilon_j \mid 1\le i<j \le n_1\}.
\end{equation*}
Since
\begin{equation*}
\gamma_{\rho_1}^+=\big(t_0t_nt^{2(n-1)}q^{n_1},\dots,t_0t_nt^{2(n-n_1)}q, t_0t_nt^{2(n_0-1)},\dots,t_0t_n\big),
\end{equation*}
by definitions of the dual multiplicity function $\tilde{\mathbf t}$, $v_\beta(x)$ and the parameters
 $\{a_1,a_2, a_3,a_4\}$ (\eqref{eq2.1}, \eqref{eq2.2}, \eqref{eq3.1}), one obtains
\begin{gather*}
\tilde v_{ - \epsilon_i+ \epsilon_j}(\gamma_{\rho_1^{+}}) = t^2 \frac{1-q^{j-i-1}t^{2(j-i)}}{1-q^{j-i}t^{2(j-i)}},
\\
\tilde v_{ - \epsilon_i- \epsilon_j}(\gamma_{\rho_1^{+}}) = t^2 \frac{1-Aq^{2n_1-i-j+1}t^{2(2n-i-j-1)}}
 {1-Aq^{2n_1-i-j+1}t^{2(2n-i-j)}},
 \\
 \tilde v_{ - 2\epsilon_i}(\gamma_{\rho_1^{+}}) =t_n^2 \frac{\big(1-q^{n_1-i+1}t^{2(n-i)}\big)\big(1-a_1a_2q^{n_1-i}t^{2(n-i)}\big)}
 {1-Aq^{2(n_1-i+1)-1}t^{4(n-i)}}.
\end{gather*}
In view of $\tilde{t}_{w_1}= t^{n_1(n_1-1)}t_{n}^{n_1}$,
these equalities conclude the proof of the desired formula~\eqref{eq4.7}.
\end{proof}

In~\eqref{eq4.7} one can apply the product formula for the Poincar\'{e} series of type C
\cite[Section 2.2]{Ma1}, which reads
\begin{align}
\sum_{w\in W_1} \tilde t_w^{-2} & = \prod_{i=1}^{n_1}\frac{\big(1-t^{-2i}\big)\big(1+t_n^{-2}t^{-2(i-1)}\big)}{1-t^{-2}} \nonumber
\\
& = t^{-2n_1(n_1-1)}(-a_3a_4)^{-n_1}\prod_{i=1}^{n_1} \frac{\big(1-t^{2i}\big)\big(1-a_3a_4t^{2(i-1)}\big)}{1-t^2}.
\label{eq4.12}
\end{align}
Note that
\begin{equation} \label{eq4.13} 
 c_{\rho_1 \rho_1}^{\dagger} = c_{\rho_1 \rho_1}.
\end{equation}

Thus far we have given an expansion of the partial antisymmetrization $U_1^{-} E_{\rho_1}(x)$
in terms of $E_{\mu}(x)$. The next step is to give an explicit formula for $U_1^{-} E_{\rho_1}(x)$.

 Put
\begin{gather*}
\chi_1(x;\mathbf t;q) = x^{\rho_1}\prod_{\alpha\in \Sigma_1^{-}}(1-x^{\alpha})v_{\alpha}\big(x;\mathbf t^{-1},q^{-1}\big)
\\ \hphantom{\chi_1(x;\mathbf t;q)}
{}= (-1)^{\frac{n_1(n_1+1)}{2}}t^{-2n_1(n_1-1)}t_n^{-2n_1}(x_{n_0+1}\cdots x_n)^{-n_1}
\\ \hphantom{\chi_1(x;\mathbf t;q)=}
 {}\times \prod_{n_0+1\le i<j\le n}\big(t^2x_i-x_j\big)\big(1-t^2x_ix_j\big)\prod_{i=n_0+1}^{n}(1-a_3x_i)(1-a_4x_i).
\end{gather*}
Then for $ f\in \mathcal A$ we have

\begin{Lemma} 
\begin{equation} \label{eq4.14} 
\big(T_i+t_i^{-1}\big)(\chi_1f)=s_i(\chi_1)(T_i-t_i)f, \qquad i=n_0+1,\dots,n.
\end{equation}
\end{Lemma}

\begin{proof}
Note first that
\begin{align*}
\big(T_i+t_i^{-1}\big)(\chi_1f) & =\big(t_i+t_i^{-1}-t_i^{-1}v_i+t_i^{-1}v_is_i\big)\chi_1f \\
 &= \big(t_i+t_i^{-1}-t_i^{-1}v_i\big)\chi_1f+s_i(\chi_1)t_i^{-1}v_is_i( f ) \\
 &= \big(t_i+t_i^{-1}-t_i^{-1}v_i\big)\chi_1f+s_i(\chi_1)\big( t_i^{-1}v_i f+(T_i-t_i))f \big).
\end{align*}
So it is sufficient to show
\begin{equation*}
\big(t_i+t_i^{-1}-t_i^{-1}v_i\big)\chi_1+t_i^{-1}v_is_i(\chi_1)=0.
\end{equation*}
Since
\begin{equation*}
\chi_1(x;\mathbf t;q)=x^{\rho_1}\prod_{\alpha\in \Sigma_1^{-}}(1-x^{\alpha})C\big(x;\mathbf t^{-1},q^{-1}\big)
\end{equation*}
and
\begin{equation*}
s_i\bigg(x^{\rho_1}\prod_{\alpha\in \Sigma_1^{-}}(1-x^{\alpha})\bigg)
=-x^{\rho_1}\prod_{\alpha\in \Sigma_1^{-}}(1-x^{\alpha}),
\end{equation*}
it follows that
\begin{equation*}
\frac{s_i(\chi_1)}{\chi_1}=-\frac{s_i\big(C\big(x;\mathbf t^{-1},q^{-1}\big) \big)}{C\big(x;\mathbf t^{-1},q^{-1}\big) }.
\end{equation*}
Here observe that
\begin{equation*}
C\big(x;\mathbf t^{-1},q^{-1}\big)=\prod_{\beta \in \Sigma^{+}}t^{-2\beta}v_{\beta}(x;\mathbf t,q),
\end{equation*}
so one clearly has
\begin{equation*}
\frac{s_i\big(C\big(x;\mathbf t^{-1},q^{-1}\big) \big)}{C\big(x;\mathbf t^{-1},q^{-1}\big) }=\frac{s_i(v_i)}{v_i},
\end{equation*}
where $ v_i= v_{a_i} (x;\mathbf t,q) $. It thus remains to be shown that
\begin{equation*}
\frac{s_i(v_i)}{v_i}=\frac{t_i+t_i^{-1}-t_i^{-1}v_i}{t_i^{-1}v_i},
\end{equation*}
and this is deduced from the definition of $v_i$.
\end{proof}

Let $\mathcal A_1$ denote the subalgebra of $W_1$-invariants of $\mathcal A$.
\begin{Proposition}
Let $ f\in \mathcal A$. Then $\big(T_i+t_i^{-1}\big)f=0$ for $n_0+1\le i\le n$ if and only if $f\in \chi_1\mathcal A_1$.
\end{Proposition}

\begin{proof}
If $\big(T+t_i^{-1}\big)f=0$, $n_0+1\le i\le n$, then~\eqref{eq4.14} implies that $g=\chi_1^{-1}f$ is killed by each $T_i-t_i$.
Hence $g$ is $ W_1$-symmetric, and so $ \chi_1^{-1}f=w_1\big(\chi_1^{-1}f\big)$ or $\ w_1(\chi_1)f=\chi_1 w_1(f)$
for any $w_1\in W_1$. This shows that $\chi_1$ divides $f$ in $\mathcal A$ since $ \chi_1$ and
$ w_1(\chi_1)$ are coprime except for the powers of $x_{n_0+1},\dots,x_n$.
Hence $g\in \mathcal A_1$. Conversely, if $ f=\chi_1 g$ with $g\in \mathcal A_1$, then~\eqref{eq4.14} shows that
$\big(T_i+t_i^{-1}\big)f=0$.
\end{proof}

\begin{Proposition}
\begin{equation} \label{eq4.15} 
U_1^{-} E_{\rho_1}(x)= c_{\rho_1 \rho_1} \chi_1(x;\mathbf t;q).
\end{equation}
\end{Proposition}

\begin{proof}
Equation~\eqref{eq4.2} implies that $\big(T_i+t_i^{-1}\big)U_1^{-} E_{\rho_1}=0$ for $n_0+1\le i \le n $,
and hence $U_1^{-} E_{\rho_1}=\chi_1g $
for some $g\in \mathcal A_1$. Note that the leading term of $U_1^{-} E_{\rho_1}$ is that of $ E_{\rho_1} $
except for its coefficient because $U_1^{-} E_{\rho_1}$ is a linear combination of $ E_{w(\rho_1)}$, $w\in W_1$.
But the leading term of~$\chi_1$ is also of the same degree. Since~$g$ is~$W_1$ invariant, it must be a constant. Finally equating
coefficients of leading terms establishes~\eqref{eq4.15}.
\end{proof}
We now calculate the norm $ \langle U_1^{-} E_{\rho_1},U_1^{-} E_{\rho_1}\rangle $.
\begin{align*}
 \langle U_1^{-} E_{\rho_1},U_1^{-} E_{\rho_1}\rangle & = \bigl<{U_1^{-}}^2 E_{\rho_1}, E_{\rho_1}\bigr> =\langle U_1^{-}
 E_{\rho_1}, E_{\rho_1}\rangle \\
 & = \biggl<\sum_{\mu\in W_1\lambda} c_{\rho_1 \mu}E_{\mu}, E_{\rho_1} \biggr> = c_{\rho_1 \rho_1} \langle E_{\rho_1}, E_{\rho_1}\rangle .
\end{align*}
Hence from~\eqref{eq4.15}
\begin{equation*}
\langle c_{\rho_1\rho_1} \,\chi_1(x;\mathbf t;q), c_{\rho_1\rho_1} \,
\chi_1(x;\mathbf t;q)\rangle = c_{\rho_1 \rho_1} \langle E_{\rho_1}, E_{\rho_1}\rangle .
 \end{equation*}
Thus, in view of~\eqref{eq4.13}, we arrive at
\begin{Theorem} \label{theorem4.6} 
\begin{equation} \label{eq4.16} 
\langle \chi_1,\chi_1\rangle
= \frac{1}{c_{\rho_1 \rho_1} } \langle E_{\rho_1}, E_{\rho_1}\rangle .
\end{equation}
\end{Theorem}
\noindent
Here, by virtue of~\eqref{eq3.6}, \eqref{eq4.7} and~\eqref{eq4.12}, the right-hand side is written explicitly as
 \begin{gather}
\frac{1}{c_{\rho_1 \rho_1}} \langle E_{\rho_1}, E_{\rho_1}\rangle = t^{-2n_1(n_1-1)}(-a_3a_4)^{-n_1}
\prod_{i=1}^{n_1}\frac{\big(1-t^{2i}\big)\big(1-a_3a_4t^{2(i-1)}\big)}{1-t^2} \nonumber
\\ \hphantom{\frac{1}{c_{\rho_1 \rho_1}} \langle E_{\rho_1}, E_{\rho_1}\rangle = }
{}\times \prod_{i=1}^{n_1} \frac{\big(Aq^{2(n_1-i+1)-1}t^{4(n-i)},Aq^{2(n_1-i+1)}t^{4(n-i)}\big)_{\infty}}
 {\big(q^{n_1-i+1}t^{2(n-i)},Aq^{n_1-i+1}t^{2(n-i)}\big)_{\infty}} \nonumber
 \\ \hphantom{\frac{1}{c_{\rho_1 \rho_1}} \langle E_{\rho_1}, E_{\rho_1}\rangle = \times \prod_{i=1}^{n_1}}
{}\times \frac{\big(1-a_3a_4q^{n_1-i+1}t^{2(n-i)}\big)} {\big(1-a_1a_2q^{n_1-i}t^{2(n-i)}\big) \prod_{1\le r<s\le 4}
 \big(a_ra_sq^{n_1-i+1}t^{2(n-i)}\big)_{\infty}} \nonumber
 \\ \hphantom{\frac{1}{c_{\rho_1 \rho_1}} \langle E_{\rho_1}, E_{\rho_1}\rangle = }
 {}\times \prod_{i=n_1+1}^n \frac{ \big(At^{4(n-i)}\big)_{\infty}^2\big(1-a_3a_4t^{2(n-i)}\big)}{\big(qt^{2(n-i)},At^{2(n-i)}\big)_{\infty}
 \prod_{1\le r<s\le 4}\big(a_ra_st^{2(n-i)}\big)_{\infty}} \nonumber
 \\ \hphantom{\frac{1}{c_{\rho_1 \rho_1}} \langle E_{\rho_1}, E_{\rho_1}\rangle = }
 {}\times \prod_{1\le i<j\le n_1}
 \frac{\big(q^{j-i}t^{2(j-i)},q^{j-i+1}t^{2(j-i)}\big)_{\infty}} {\big(q^{j-i+1}t^{2(j-i+1)},q^{j-i}t^{2(j-i-1)}\big)_{\infty}} \nonumber
 \\ \hphantom{\frac{1}{c_{\rho_1 \rho_1}} \langle E_{\rho_1}, E_{\rho_1}\rangle =\times \prod_{1\le i<j\le n_1} }
{}\times \frac{\big(Aq^{2n_1-i-j+1}t^{2(2n-i-j)},Aq^{2n_1-i-j+2}t^{2(2n-i-j)}\big)_{\infty}}
 {\big(Aq^{2n_1-i-j+2}t^{2(2n-i-j+1)},Aq^{2n_1-i-j+1}t^{2(2n-i-j-1)}\big)_{\infty}} \nonumber
 \\ \hphantom{\frac{1}{c_{\rho_1 \rho_1}} \langle E_{\rho_1}, E_{\rho_1}\rangle = }
 {}\times \prod_{i=1}^{n_1} \prod_{j=n_1+1}^{n}
 \frac{\big(q^{n_1-i+1}t^{2(j-i)}\big)_{\infty}^2} {(q^{n_1-i+1}t^{2(j-i+1)},q^{n_1-i+1}t^{2(j-i-1)})_{\infty}}\nonumber
 \\ \hphantom{\frac{1}{c_{\rho_1 \rho_1}} \langle E_{\rho_1}, E_{\rho_1}\rangle = \times \prod_{i=1}^{n_1}\prod_{j=n_1+1}^{n}}
 {}\times \frac{\big(Aq^{n_1-i+1}t^{2(2n-i-j)}\big)_{\infty}^2}{\big(Aq^{n_1-i+1}t^{2(2n-i-j+1)},
 Aq^{n_1-i+1}t^{2(2n-i-j-1)}\big)_{\infty}} \label{eq4.17}
 \\ \hphantom{\frac{1}{c_{\rho_1 \rho_1}} \langle E_{\rho_1}, E_{\rho_1}\rangle = }
 {}\times \prod_{n_1+1\le i<j\le n} \frac{\big(qt^{2(j-i)},At^{2(2n-i-j)}\big)_{\infty}^2}{\big(qt^{2(j-i+1)},qt^{2(j-i-1)},
 At^{2(2n-i-j+1)},At^{2(2n-i-j-1)}\big)_{\infty}},\nonumber
 \end{gather}
and the left-hand side is
\begin{align*}
\langle \chi_1,\chi_1\rangle ={}& \frac{1}{(2\pi {\rm i})^n}\int_{ C^n} \prod_{i=n_0+1}^{n}(1-a_3x_i)(1-a_4x_i)
\big(1-a_3^{-1}x_i^{-1}\big)\big(1-a_4^{-1}x_i^{-1}\big)
\\
&\times \prod_{n_0+1\le i<j\le n}\big(t^2x_i-x_j\big)\big(1-t^2x_ix_j\big) \big(t^{-2}x_i^{-1}-x_j^{-1}\big)\big(1-t^{-2}x_i^{-1}x_j^{-1}\big)
\Delta(x) \frac{{\rm d}x}{x}.
\end{align*}
This theorem may be viewed as a variation of the Selberg-type integral formula due to Gustaf\-son~\cite{G},
in which the integral $\langle 1,1\rangle $ was evaluated.
We will prove that, by a suitable limiting procedure, the formula~\eqref{eq4.16} gives a variation of the
$q$-Selberg integral formula~\cite{As,Hab,Kad} involving a $q$-difference product with respect to part of the variables,
see Theorem~\ref{theorem7.3}.

Before proceeding to the limiting procedure, we first observe that~\eqref{eq4.17} can be simplified
using the identities
\begin{gather*}
\prod_{1\le i<j \le n} \frac{1-QT^{j-i+1}}{1-QT^{j-i}} =\prod_{i=1}^{n} \frac{1-QT^i}{1-QT}, \qquad
 \prod_{1\le i<j \le n} \frac{1-QT^{i+j+1}}{1-QT^{i+j}} = \prod_{i=1}^{n} \frac{1-QT^{2i}}{1-QT^{i+1}}.
\end{gather*}
In fact the following formulae immediately follow from these identities:
 \begin{gather*}
 \prod_{1\le i < j \le n_1}
 \frac{\big(q^{j-i}t^{2(j-i)},q^{j-i+1}t^{2(j-i)} \big)_{\infty}}{\big(q^{j-i+1}t^{2(j-i+1)},q^{j-i}t^{2(j-i-1)} \big)_{\infty}}
= \prod_{i=1}^{n_1} \frac{\big(qt^2, q^it^{2(i-1)}\big)_{\infty}}{(q^it^{2i}, q)_{\infty}} ,
\\
\prod_{n_1+1\le i< j\le n} \frac{\big(qt^{2(j-i)}\big)_{\infty}^2}{\big(qt^{2(j-i+1)},qt^{2(j-i-1)}\big)_{\infty}}
= \prod_{i=1}^{n_0} \frac{\big(qt^2, qt^{2(i-1)}\big)_{\infty}}{(qt^{2i}, q)_{\infty}} ,
\\
 \prod_{1\le i < j \le n_1} \frac{\big(Aq^{2n_1-i-j+1}t^{2(2n-i-j)},Aq^{2n_1-i-j+2}t^{2(2n-i-j)}\big)_{\infty}}
{\big(Aq^{2n_1-i-j+2}t^{2(2n-i-j+1)}Aq^{2n_1-i-j+1}t^{2(2n-i-j-1)}\big)_{\infty}}
\\ \qquad
{}= \prod_{i=1}^{n_1}\frac{\big(Aq^{i}t^{2(2n_0+i-1)},Aq^{2i-1}t^{2(2n_0+2i-3)}\big)_{\infty}} {\big(A^{2i-1}t^{4(n_0+i-1)},Aq^{i}t^{2(2n_0+i-2)}\big)_{\infty}},
\\
\prod_{n_1+1\le i < j\le n} \frac{\big(At^{2(2n-i-j)}\big)_{\infty}^2}{\big(At^{2(2n-i-j+1)},At^{2(2n-i-j-1)}\big)_{\infty}}
= \prod_{i=1}^{n_0} \frac{\big(At^{2(i-1)}\big)_{\infty}}{\big(At^{4(i-1)}\big)_{\infty}}
 \prod_{i=1}^{n_0}\frac{\big(At^{2(2i-3)}\big)_{\infty}}{\big(At^{2(i-2)}\big)_{\infty}}
 \\ \qquad
{} =\frac{\big(At^{2(n_0-1)}\big)_{\infty}}{(A)_{\infty}}\prod_{i=2}^{n_0}
\frac{\big(At^{2(2i-3)}\big)_{\infty}}{\big(At^{4(i-1)}\big)_{\infty}}
=\frac{\big(At^{2(n_0-1)}\big)_{\infty}}{\big(At^{2(2n_0-1)}\big)_{\infty}}\prod_{i=1}^{n_0}
\frac{\big(At^{2(2i-1)}\big)_{\infty}}{\big(At^{4(i-1)}\big)_{\infty}}.
\end{gather*}
We rewrite the product $\prod_{i=n_1+1}^{n} $ in~\eqref{eq4.17} as the product $\prod_{i=1}^{n_0} $. The product
$\prod_{i=1}^{n_1} \prod_{j=n_1+1}^{n} $ also can be rewritten as the product $\prod_{i=1}^{n_1}$:
\begin{gather*}
 \prod_{i=1}^{n_1} \prod_{j=n_1+1}^{n}
\frac{\big(q^{n_1-i+1}t^{2(j-i)}\big)_{\infty}^2} {\big(q^{n_1-i+1}t^{2(j-i+1)},q^{n_1-i+1}t^{2(j-i-1)}\big)_{\infty}}
\\ \qquad
{} = \prod_{i=1}^{n_1}
 \frac{\big(q^{n_1-i+1}t^{2(n_1+1-i)},q^{n_1-i+1}t^{2(n-i)}\big)_{\infty}} {\big(q^{n-i+1}t^{2(n-i+1)},q^{n_1-i+1}t^{2(n_1-i)}\big)_{\infty}}
 = \prod_{i=1}^{n_1}
 \frac{\big(q^{i}t^{2i},q^{i}t^{2(n_0+i-1)}\big)_{\infty}} {\big(q^{i}t^{2(i-1)},q^{i}t^{2(n_0+i)}\big)_{\infty}},
 \end{gather*}
and
\begin{gather*}
\prod_{i=1}^{n_1} \prod_{j=n_1+1}^{n} \frac{\big(Aq^{n_1-i+1}t^{2(2n-i-j)}\big)_{\infty}^2}
{\big(Aq^{n_1-i+1}t^{2(2n-i-j+1)},Aq^{n_1-i+1}t^{2(2n-i-j-1)}\big)_{\infty}}
\\ \qquad
{}= \prod_{i=1}^{n_1} \prod_{j=1}^{n_0} \frac{\big(Aq^{i}t^{2(n_0+i+j-2)}\big)_{\infty}^2}
{\big(Aq^{i}t^{2(n_0+i+j-1)},Aq^{i}t^{2(n_0+i+j-3)}\big)_{\infty}}
\\ \qquad
{}=\prod_{i=1}^{n_1} \frac{\big(Aq^{i}t^{2(n_0+i-1)},Aq^{i}t^{2(2n_0+i-2)} \big)_{\infty}}
{\big(Aq^{i}t^{2(n_0+i-2)},Aq^{i}t^{2(2n_0+i-1)}\big)_{\infty}}.
 \end{gather*}
Substituting these formulae into~\eqref{eq4.17}, after some cancellations, we obtain
\begin{equation} \label{eq4.18} 
\frac{1}{c_{\rho_1 \rho_1} } \langle E_{\rho_1}, E_{\rho_1}\rangle =F(a_1,a_2, a_3,a_4)\,G(a_1,a_2,a_3,a_4),
\end{equation}
where
\begin{gather}
F(a_1,a_2, a_3,a_4) \nonumber
\\ \qquad
 {}=(-a_3a_4)^{-n_1}\prod_{i=1}^{n_1}
\frac{\big(1-a_3a_4t^{2(i-1)}\big)\big(1-a_3a_4q^{{n_1-i+1}}t^{2(n-i)}\big)} {1-a_1a_2q^{{n_1-i}}t^{2(n-i )} } \prod_{i=1}^{n_0}\big(1-a_3a_4t^{2(i-1)}\big) \nonumber
 \\ \qquad
{}=\prod_{i=1}^{n_1}\frac{\big(1-a_3a_4t^{2(i-1)}\big)\big(1-a_3a_4q^{{i}}t^{2(n_0+i-1)}\big)} {-a_3a_4+Aq^{{i-1}}
t^{2(n_0+i-1 )} } \prod_{i=1}^{n_0}\big(1-a_3a_4t^{2(i-1)}\big),\label{eq4.19}
\end{gather}
and
\begin{align*}
 G(a_1,a_2,a_3,a_4)
 ={}& t^{-2n_1(n_1-1)}\frac{\big(qt^2\big)_{\infty}^n}{(q)_{\infty}^n} \prod_{i=1}^{n_0} \frac{1}{\big(qt^{2i}\big)_{\infty}}
\prod_{i=1}^{n_1} \frac{1-t^{2i}}{\big(1-t^2\big)\big(q^i t^{2(n_0+i)}\big)_{\infty}}
\\[1ex]
&\times \frac{\big(At^{2(n_0-1)}\big)_{\infty}}{\big(At^{2(2n_0-1)}\big)_{\infty}}\prod_{i=1}^{n_0}
\frac{\big(At^{2(2i-1)},At^{4(i-1)}\big)_{\infty}}
{\big(At^{2(i-1)}\big)_{\infty} \prod_{1\le r<s\le 4}\big(a_ra_st^{2(i-1)}\big)_{\infty}}
\\[1ex]
&\times \prod_{i=1}^{n_1}\frac{\big(Aq^{2i-1}t^{2(2n_0+2i-3)},Aq^{2i}t^{4(n_0+i-1)}\big)_{\infty}}
{\big(Aq^{i}t^{2(n_0+i-2)}\big)_{\infty}\prod_{1\le r<s\le 4}
 \big(a_ra_sq^{i}t^{2(n_0+i-1)}\big)_{\infty} }.
\end{align*}
 But since
 \begin{gather*} \frac{\big(At^{2(n_0-1)}\big)_{\infty}}{\big(At^{2(2n_0-1)}\big)_{\infty}} \prod_{i=1}^{n_0}\frac{\big(At^{2(2i-1)},At^{4(i-1)}\big)_{\infty}}
{\big(At^{2(i-1)}\big)_{\infty} }= \prod_{i=1}^{n_0} \big(At^{2(n_0+i-2)}\big)_{\infty},
\\
 \prod_{i=1}^{n_1}\big(Aq^{2i-1}t^{2(2n_0+2i-3)},Aq^{2i}t^{4(n_0+i-1)}\big)_{\infty}
=\prod_{i=1}^{2n_1} \big(Aq^i t^{2(2n_0+i-2)}\big)_{\infty},
\end{gather*}
we find that
 \begin{align}
 G(a_1,a_2,a_3,a_4)
= {}&t^{-2n_1(n_1-1)}\frac{(qt^2)_{\infty}^n}{(q)_{\infty}^n} \prod_{i=1}^{n_0} \frac{1}{(qt^{2i})_{\infty}}
\prod_{i=1}^{n_1} \frac{1-t^{2i}}{(1-t^2)\big(q^i t^{2(n_0+i)}\big)_{\infty}} \nonumber
\\
& \times \prod_{i=1}^{n_0}
\frac{\big(At^{2(n_0+i-2)}\big)_{\infty}}
{ \prod_{1\le r<s\le 4}\big(a_ra_st^{2(i-1)}\big)_{\infty}} \nonumber
\\
&\times\frac{\prod_{i=1}^{2n_1} \big(Aq^i t^{2(2n_0+i-2)}\big)_{\infty}}{\prod_{i=1}^{n_1}
\big(Aq^{i}t^{2(n_0+i-2)}\big)_{\infty}\prod_{1\le r<s\le 4}
 \big(a_ra_sq^{i}t^{2(n_0+i-1)}\big)_{\infty} }. \label{eq4.20}
\end{align}
We have thus established the formula~\eqref{eq4.16} in a simplified form:
 \begin{gather}
 \frac{1}{(2\pi {\rm i})^n}\int_{ C^n} \prod_{i=n_0+1}^{n}(1-a_3x_i)(1-a_4x_i)
\big(1-a_3^{-1}x_i^{-1}\big)\big(1-a_4^{-1}x_i^{-1}\big) \nonumber
\\ \qquad
 {}\times \prod_{n_0+1\le i<j\le n}\big(t^2x_i-x_j\big)\big(1-t^2x_ix_j\big) \big(t^{-2}x_i^{-1}-x_j^{-1}\big)\big(1-t^{-2}x_i^{-1}x_j^{-1}\big)
\Delta(x) \frac{{\rm d}x}{x} \nonumber
\\ \qquad
{} =t^{-2n_1(n_1-1)}(-a_3a_4)^{-n_1}\frac{\big(qt^2\big)_{\infty}^n}{(q)_{\infty}^n}
 \prod_{i=1}^{n_0} \frac{1}{\big(qt^{2i}\big)_{\infty}}
\prod_{i=1}^{n_1} \frac{1-t^{2i}}{\big(1-t^2\big)\big(q^i t^{2(n_0+i)}\big)_{\infty}} \nonumber
\\ \qquad\hphantom{=}
{}\times \prod_{i=1}^{n_0}\big(1-a_3a_4t^{2(i-1)}\big)
 \prod_{i=1}^{n_1}
 \frac{\big(1-a_3a_4t^{2(i-1)}\big)\big(1-a_3a_4q^{{n_1-i+1}}t^{2(n-i)}\big)} {1-a_1a_2q^{{n_1-i}}t^{2(n-i )}\nonumber }
\\ \qquad\hphantom{=}
{}\times \prod_{i=1}^{n_0} \frac{\big(At^{2(n_0+i-2)}\big)_{\infty}}
{ \prod_{1\le r<s\le 4}\big(a_ra_st^{2(i-1)}\big)_{\infty}}\nonumber
\\ \qquad\hphantom{=}
{}\times\frac{\prod_{i=1}^{2n_1} \big(Aq^i t^{2(2n_0+i-2)}\big)_{\infty}}{\prod_{i=1}^{n_1}
\big(Aq^{i}t^{2(n_0+i-2)}\big)_{\infty}\prod_{1\le r<s\le 4}
 \big(a_ra_sq^{i}t^{2(n_0+i-1)}\big)_{\infty} }.\label{eq4.21}
\end{gather}

\section{Limit of the quadratic norm formula: Part 1} \label{section5}
 In this section we calculate the limit of the left-hand side of the quadratic norm formula~\eqref{eq4.21}.

Hereafter we assume that
 \begin{equation*}
 t^2= q^k, \qquad k \in \mathbb Z_{\ge 0}.
 \end{equation*}
Put
 \begin{equation*}
 \begin{split}
& \text{\emph{Case} I:} \quad\ \ (a_1,a_2,a_3,a_4)
 = \big(\epsilon^{-1}q^{\frac{1}{2}},-aq^{\frac{1}{2}},\epsilon bq^{\frac{1}{2}}, -q^{\frac{1}{2}}\big),
 \\
& \text{\emph{Case} II:} \quad \ (a_1,a_2,a_3,a_4)
 = \big( \epsilon bq^{\frac{1}{2}} ,-aq^{\frac{1}{2}},\epsilon^{-1}q^{\frac{1}{2}}, -q^{\frac{1}{2}}\big),
\end{split}
 \end{equation*}
 where $0<|a|, |b|<1$. We shall also write $a=q^{\alpha}$, $b=q^{\beta}$. First $\epsilon$ is taken so that
 $|a_i|<1$, $i=1,2,3,4$ and then,
after multiplying the norm formula~\eqref{eq4.16} by a suitable factor, we take the limit $ \epsilon \to 0$,
or more precisely we take a suitable sequence $\{\epsilon_r\}$ which converges to $0$.
For this purpose, define
 \begin{equation*}
 \epsilon_r=\epsilon_0 q^r, \qquad r\ge 0.
 \end{equation*}
 We assume $q^{\frac{1}{2}}<|\epsilon_0|<q^{-\frac{1}{2}}$ so that, for $0\le s \le r-1$, one has
 \begin{equation} \label{eq5.1} 
 q^{r+\frac{1}{2}}< |\epsilon_r |<q^{s+\frac{1}{2}}.
 \end{equation}

Let us consider the integrand $ \chi_1(x) \chi_1(x) ^{\dagger} \Delta(x)\frac{1}{x} $.
We put
\begin{gather*}
 g(x) = \prod_{i=n_0+1}^{n} (1-a_3x_i)(1-a_4x_i)\big(1-a_3^{-1}x_i^{-1}\big)\big(1-a_4^{-1}x_i^{-1}\big)\\
 \hphantom{g(x) =}{}\times
 \prod_{i=1}^n (x_i-a_3)(x_i-a_4)\big(x_i^{-2}-1\big),
 \\
h(x) = \prod_{n_0+1\le i<j\le n}\big(t^2x_i-x_j\big)\big(1-t^2x_ix_j\big) \big(t^{-2}x_i^{-1}-x_j^{-1}\big)\big(1-t^{-2}x_i^{-1}x_j^{-1}\big)
\\ \hphantom{h(x) =}
{}\times\prod_{1\leq i<j \leq n}\big(t^2-x_ix_j\big)\big(t^2-x_ix_j^{-1}\big)\big(1-x_i^{-1}x_j\big)\big(1-x_i^{-1}x_j^{-1}\big),
 \\
 D(x) = \prod_{1\leq i<j \leq n}\frac{\big(qx_ix_j,qx_ix_j^{-1},qx_i^{-1}x_j,qx_i^{-1}x_j^{-1}\big)_{\infty}}
 {\big(t^2x_ix_j,t^2x_ix_j^{-1},t^2x_i^{-1}x_j,t^2x_i^{-1}x_j^{-1}\big)_{\infty}}
 \\ \hphantom{ D(x)}
{} = \prod_{1\leq i<j \leq n} \big(qx_ix_j,qx_ix_j^{-1},qx_i^{-1}x_j,qx_i^{-1}x_j^{-1}\big)_{k-1},
 \\
 A(x_i) = \frac{\big(qx_i^2,qx_i^{-2}\big)_{\infty}}{x_i\prod_{r=1}^4\big(a_rx_i,a_rx_i^{-1}\big)_{\infty}},
\end{gather*}
so that
\begin{equation*}
 \chi_1(x) \chi_1(x) ^{\dagger} \Delta(x)\frac{1}{x}=g(x)h(x) D(x) \prod_{i=1}^{n}A(x_i).
\end{equation*}
In either case, \emph{Case} I or \emph{Case} II, one has
\begin{equation*}
A(x_i)=\frac{\big(qx_i,-qx_i,q^{\frac{1}{2}}x_i,qx_i^{-1},-qx_i^{-1}, q^{\frac{1}{2}}x_i^{-1}\big)_{\infty}}
{x_i\big(\epsilon^{-1}q^{\frac{1}{2}}x_i, \epsilon^{-1}q^{\frac{1}{2}}x_i^{-1},
-aq^{\frac{1}{2}}x_i,-aq^{\frac{1}{2}}x_i^{-1},\epsilon bq^{\frac{1}{2}}x_i, \epsilon bq^{\frac{1}{2}}x_i^{-1} \big)_{\infty}}.
\end{equation*}
Note that our integrand is a product of the Laurent polynomial $ g(x)h(x) D(x) $ and functions of one variable
 $\prod_{i=1}^{n}A(x_i)$.

For $\epsilon=\epsilon_r$, $A(x_i)$ has a non-isolated singularity at the origin and has poles
at the points
\begin{equation*}
x_{r,s}=\epsilon_r^{-1} q^{s+\frac{1}{2}}, \qquad
y_{r,s'} = \epsilon_{r} bq^{s'+\frac{1}{2}}, \qquad
z_{s''} =-a q^{s''+\frac{1}{2}}
\end{equation*}
and the reciprocals of these points, where $s,s',s'' \in \mathbb Z_{\ge 0}$.
We assume that $\epsilon_0 $ is generic so that $ x_{r,s},\dots, z_{s''}^{-1} $
are distinct and do not lie on the unit circle, so that the poles of each $A_i(x)$
are all simple. Note also that, by \eqref{eq5.1}, $x_{0,s}$ (resp.\ $x_{0,s}^{-1}$) lies in the interior
(resp.\ exterior) of the unit circle.

When we take the limit $\epsilon_r \to 0$, some poles may move from the interior (resp.\ exterior) to
the exterior (resp.\ exterior) of the unit circle. More precisely, while $y_{r,s'}$ and~$z_{s''}$
(resp.\ $y_{r,s''}^{-1}$ and~$z_{s''}^{-1}$)
 always lie in the interior (resp.\ exterior) of the unit circle, $x_{r,s}$ (resp.\ $ x_{r,s}^{-1} $) moves from
 the interior (resp.\ exterior) to the exterior (resp.\ exterior) of
the unit circle as $r$ increases. In~fact $ {x_{r,s}}'s$ (resp.\ $ {x_{r,s}^{-1}}'s$) lie in the interior (resp.\ exterior)
of the unit circle for $0\le s \le r-1$ in view of~\eqref{eq5.1}. Hence we must deform the unit circle in
each coordinate complex plane to include $ x_{r,s}$ and exclude $ x_{r,s}^{-1} $
 keeping $y_{r,s'}$, $z_{s''}$ inside and $y_{r,s'}^{-1}$, $z_{s''}^{-1}$ outside.
Accordingly, we define the deformed contour $C_{r}$ for $r$ as a closed curve enclosing the origin,
$x_{r,s}$, $y_{r,s'}$, $z_{s''}$ and excluding $x_{r,s}^{-1}$, $y_{r,s'}^{-1}$, $z_{s''}^{-1}$ for any
$ s,s', s'' \in \mathbb Z_{\ge 0}$. Let $C_{r,1}$ (resp.\ $C_{r,-1}$) be a positively oriented circle encircling only
$x_{r,s} $ (resp.\ $x_{r,s}^{-1}$) with $0\le s \le r-1$. Then in each coordinate plane one has
 \begin{equation*}
 C_{r}=C +C_{r,1}-C_{r,-1}
 \end{equation*}
as the path of integration. Let us write
\begin{equation*}
\mathrm{sgn}(x)=
\begin{cases}
\phantom{-}1, & x\geq 0, \\
-1, & x<0.
\end{cases}
\end{equation*}
We will consider the limit of the product of a suitable factor and the integral
\begin{equation} \label{eq5.2} 
\frac{}{}\int_{C_r^n} \chi_1(x) \chi_1(x) ^{\dagger} \Delta(x) \frac{{\rm d}x}{x}
=\sum_{
\substack{\kappa_i=0,\pm1 \\
i=1,\dots,n} }
\prod_{i=1}^n \mathrm{sgn}(\kappa_i)\int_{C_{r,\kappa_1}\times \cdots \times C_{r,\kappa_n}}
 \chi_1(x) \chi_1(x) ^{\dagger} \Delta(x) \frac{{\rm d}x}{x},
 \end{equation}
where $ C_{r,0}$ should be understood as the unit circle $C$. We shall prove later that the terms with
$\kappa_i=0$ for some $ i $ is negligible in the limit calculation.

It is immediate to see that
\begin{equation*}
\Res\limits_{x_i=x_{r,s}^{-1}} A(x_i) = -\Res\limits_{x_i=x_{r,s}}A(x_i),
\end{equation*}
and hence
\begin{equation*}
\mathrm{sgn}(\kappa_i)\Res\limits_{x_i=x_{r,s}^{\kappa_i}} A(x_i) = \Res\limits_{x_i=x_{r,s}} A(x_i).
\end{equation*}
Therefore, for $\kappa_i=\pm1$, $i=1,\dots,n$, we have
 \begin{gather} 
\frac{1}{(2\pi {\rm i})^n}\prod_{i=1}^n \mathrm{sgn}(\kappa_i)
\int_{C_{r,\kappa_1}\times \cdots \times C_{r,\kappa_n}}
 \chi_1(x) \chi_1(x) ^{\dagger} \Delta(x) \frac{{\rm d}x}{x}\nonumber
 \\ \qquad
 {}=\sum_{\substack{0\le s_i\le r-1 \nonumber
 \\
 i=1,\dots,n}} g\big(x_{r,s_1}^{\kappa_1}, \dots, x_{r,s_n}^{\kappa_n} \big)
 h\big(x_{r,s_1}^{\kappa_1}, \dots, x_{r,s_n}^{\kappa_n} \big)
 D\big(x_{r, s_1}^{\kappa_1},\dots,x_{r,s_n}^{\kappa_n}\big)\nonumber
 \\ \qquad\hphantom{=}
{}\times\prod_{i=1}^{n} \Res\limits_{x_i=x_{r,s_i}} A(x_i).\label{eq5.3}
\end{gather}
Here the residue $ \Res \limits_{x_i=x_{r,s}} A(x_i) $ can be easily shown to be
\begin{gather} 
 \Res\limits_{x_i=x_{r,s}} A(x_i)\nonumber
 \\ \qquad
{} = \frac{(-1)^sq^{\frac{s(s+1)}{2}}}{(q)_{\infty}(q)_{s}}
 \frac{\big(\epsilon_r^{-1}q^{s+\frac{3}{2}}, -\epsilon_r^{-1}q^{s+\frac{3}{2}}
 ,\epsilon_r^{-1}q^{s+1} ,\epsilon_rq^{-s+\frac{1}{2}} , -\epsilon_rq^{-s+\frac{1}{2}} , \epsilon_rq^{-s} \big)_{\infty}}
{\big(\epsilon_r^{-2}q^{s+1}, -\epsilon_r^{-1}a q^{s+1} , -\epsilon_r a q^{-s} , bq^{s+1} , \epsilon_{r}^2 bq^{-s}\big)_{\infty}}.\label{eq5.4}
\end{gather}
Since $D(x)$ is $W$-invariant, we have
\begin{align} 
D\big(x_{r, s_1}^{\kappa_1},\dots,x_{r,s_n}^{\kappa_n}\big)& =D(x_{r, s_1},\dots,x_{r,s_n}) \nonumber
\\
& = \prod_{1\le i < j\le n}\big( \epsilon_r^{-2}q^{s_i+s_j+2} , q^{s_i-s_j+1}, q^{-s_i+s_j+1}, \epsilon_r^2q^{-s_i-s_j}\big)_{k-1}.\label{eq5.5}
\end{align}
For the calculation of the limit of the norm, we first consider the limits of $D(x)$, $h(x)$ with $x=(x_{r, s_1}^{\kappa_1},
\dots,x_{r,s_n}^{\kappa_n})$ and $\Res\limits_{x_i=x_{r,s_i}} A(x_i) $.
Hereafter we denote by $K$ a positive constant independent of $ r$ and $s$.
For $\kappa=(\kappa_1,\dots,\kappa_n)$, put
$d(\kappa)=\sum_{1\le i< j\le n}(1-\kappa_i)=\frac{n(n-1)}{2}-\sum_{i=1}^n(n-i)\kappa_i$.
\begin{Lemma} \label{lemma5.1} 
The following inequalities hold:
\begin{gather}
\big|\epsilon_r^{n(n-1)(k-1)}D\big(x_{r, s_1}^{\kappa_1},\dots,x_{r,s_n}^{\kappa_n}\big)\big|
 \le K, \notag \\
\big|\epsilon_{r}^{n(n-1)+n_1(n_1-1)}h\big(x_{r, s_1}^{\kappa_1},\dots,x_{r,s_n}^{\kappa_n}\big)\big| \le K.
 \label{eq5.6} 
\end{gather}
Moreover we have
\begin{gather*}
\lim_{r\to \infty} \epsilon_r^{n(n-1)(k-1)}D\big(x_{r, s_1}^{\kappa_1},\dots,x_{r,s_n}^{\kappa_n}\big)
\\ \qquad
{}=(-1)^{(k-1)\binom{n}{2}}q^{(n-1)(k-1)(\Sigma_{i=1}^ns_i)+\frac{(k-1)(k+2)}{2} \binom{n}{2}}\prod_{1\le i < j\le n}
\big(q^{s_i-s_j+1}, q^{-s_i+s_j+1}\big)_{k-1},
 \end{gather*}
and
\begin{gather} 
 \lim_{r\to \infty}\,\epsilon_r^{n(n-1)+n_1(n_1-1)} h\big(x_{r,s_1}^{\kappa_1},\dots,x_{r,s_n}^{\kappa_n}\big)\nonumber
 \\ \qquad
{} =q^{\binom{n_1}{2}+\binom{n}{2}}\,t^{2d(\kappa)}
\prod_{1\le i<j\le n}\big(q^{s_i}-t^{2\kappa_i}q^{s_j}\big)\big(q^{s_i}-q^{s_j}\big)
 \prod_{n_0+1\le i<j\le n}\big(t^{2\kappa_i}q^{s_i}-q^{s_j}\big)^2.\label{eq5.7}
 \end{gather}
In particular,
\begin{gather*}
\lim_{r\to \infty}\epsilon_{r}^{n(n-1)+n_1(n_1-1)}h\big(x_{r, s_1},
\dots,x_{r,s_{n_0}},x_{r,s_{n_0+1}}^{-1},\dots,x_{r,s_n}^{-1}\big)
\\ \qquad
{} =q^{\binom{n_1}{2}+\binom{n}{2}} t^{n_1(n_1-1)}\!\!\!\!\! \prod_{n_0+1\le i<j\le n}\!\!\!\!
\big(t^{-2}q^{s_i}\!-q^{s_j}\big)\big(q^{s_i}\!-t^{-2}q^{s_j}\big)
\!\!\prod_{1\le i<j\le n}\!\!\big(q^{s_i}\!-q^{s_j}\big)\big(q^{s_i}\!-t^2q^{s_j}\big),\nonumber
\end{gather*}
and
\begin{gather*}
\lim_{r\to \infty}\epsilon_{r}^{n(n-1)+n_1(n_1-1)}h\big(x_{r, s_1}^{-1},
\dots,x_{r,s_{n_0}}^{-1},x_{r,s_{n_0+1}},\dots,x_{r,s_n}\big)
\\ \qquad
{} =q^{\binom{n_1}{2}+\binom{n}{2}} t^{n_0(n_0-1)+2n_0n_1}\!\!\!\!\!\!\!\prod_{n_0+1\le i<j\le n}
\!\!\!\!\!\!\big(t^2 q^{s_i}\!-\!q^{s_j}\big)\big(q^{s_i}\!-\!t^2 q^{s_j}\big)
\!\!\! \prod_{1\le i<j\le n}\!\!\big(q^{s_i}\!-\!q^{s_j}\big)\big(t^2q^{s_i}\!-\!q^{s_j}\big).
\end{gather*}
\end{Lemma}

\begin{proof}
The inequality and the limit formula for $D(x_{r, s_1}^{\kappa_1},\dots,x_{r,s_n}^{\kappa_n}) $
are immediate from~\eqref{eq5.5}.
The inequality~\eqref{eq5.6} is also clear from the definition of $h(x)$. For the limit~\eqref{eq5.7}, observe that
\begin{gather*}
 \lim_{r\to \infty} \epsilon_r^{2} \big(x_{r,s_i}^{\kappa_i}x_{r,s_j}^{\kappa_j}-t^2\big)
\big(x_{r,s_i}^{\kappa_i}x_{r,s_j}^{-\kappa_j}-t^2\big)\big(x_{r,s_i}^{-\kappa_i}x_{r,s_j}^{\kappa_j}-1\big)
\big(x_{r,s_i}^{- \kappa_i}x_{r,s_j}^{-\kappa_j}-1\big)
\\ \qquad
{}=qt^{1-\kappa_i} \big(t^{1-\kappa_i}q^{s_i}-t^{1+\kappa_i}q^{s_j}\big)\big(q^{s_i}-q^{s_j}\big),
\\
 \lim_{r\to \infty} \epsilon_r^{2}\big(t^2x_{r,s_i}^{\kappa_i}-x_{r,s_j}^{\kappa_j}\big)\big(1-t^2x_{r,s_i}^{\kappa_i}
 x_{r,s_j}^{\kappa_j}\big)
 \big(t^{-2}x_{r,s_i}^{-\kappa_i}-x_{r,s_j}^{-\kappa_j}\big)\big(1-t^{-2}x_{r,s_i}^{-\kappa_i}x_{r,s_j}^{-\kappa_j}\big) \\ \qquad
{}= q\big(t^{2\kappa_i}q^{s_i}-q^{s_j}\big)^2
 \end{gather*}
from which~\eqref{eq5.7} follows.
\end{proof}
\begin{Lemma}\label{lemma5.2}
The following inequality and limit formula hold:
\begin{gather}
\big|\big({-}\epsilon_r^{-1}q,-\epsilon_r^{-1}aq\big)_{\infty}\epsilon_r^{-2}\operatorname{Res}_{x_i=x_{r,s_i}}
 A(x_i)\big|\le K |a|^{s_i}q^{-(s_i+1)}, \label{eq5.8}
 \\
 \lim_{r\to \infty} \big({-}\epsilon_r^{-1}q,-\epsilon_r^{-1}aq\big)_{\infty}\epsilon_r^{-2}\operatorname{Res}_{x_i=x_{r,s_i}} A(x_i)=
-\frac{a^{s_{i}}q^{-(s_i+1)}}{\big(q,bq^{s_i+1}\big)_{\infty}(q)_{s_i}}.\label{eq5.9}
\end{gather}
\end{Lemma}

\begin{proof}
By~\eqref{eq5.4} one has
\begin{gather*}
\big({-}\epsilon_r^{-1}q,-\epsilon_r^{-1}aq\big)_{\infty}\operatorname{Res}_{x_i=x_{r,s}} A(x_i)
\\ \qquad
{}=\frac{(-1)^sq^{\frac{s(s+1)}{2}}}{(q)_{\infty}(q)_{s}}
\frac{\big({-}\epsilon_r^{-1}q,-\epsilon_r^{-1}aq\big)_s}{\big(\epsilon_r^{-2}q^{s+1}\big)_{s+1}}
\frac{\big( \epsilon_rq^{-s+\frac{1}{2}}, -\epsilon_rq^{-s+\frac{1}{2}} , \epsilon_rq^{-s}\big)_{\infty}}
{\big( {-}\epsilon_r a q^{-s} , bq^{s+1} , \epsilon_{r}^2 bq^{-s}\big)_{\infty}},
\end{gather*}
which implies the limit formula~\eqref{eq5.9}. For the inequality~\eqref{eq5.8}, we note that
\begin{equation*}
\Bigg|\frac{\big( \epsilon_rq^{-s+\frac{1}{2}}, -\epsilon_rq^{-s+\frac{1}{2}} , \epsilon_rq^{-s}\big)_{\infty}}
{\big( {-}\epsilon_r a q^{-s} , bq^{s+1} , \epsilon_{r}^2 bq^{-s}\big)_{\infty}}
 \Bigg| \le \bigg| \frac{1}{(|b|,bq^{s+1})_{\infty}} \bigg| \le \frac{1}{(|b|)_{\infty}^2}.
\end{equation*}
One also has
\begin{equation*}
\epsilon_r^{-2} \frac{\big({-}\epsilon_r^{-1}q,-\epsilon_r^{-1}aq\big)_s}{\big(\epsilon_r^{-2}q^{s+1}\big)_{s+1}}=\frac{\prod_{l=1}^s
 \big(\epsilon_r+q^l \big)\big( \epsilon_r+aq^l \big)}{\prod_{l=1}^{s+1} \big(\epsilon_{r}^2 -q^{s+l} \big)},
\end{equation*}
and so
\begin{gather*}
 \bigg|\prod_{l=1}^s \big(\epsilon_r+q^l \big)\big( \epsilon_r+aq^l \big)\bigg |\le
 \bigg|\prod_{l=1}^{s}\big(q^{s+\frac{1}{2}}+q^l\big)\big(q^{s+\frac{1}{2}}+|a|q^l\big)\bigg|
 \le
|a|^{s}q^{s(s+1)}\big({-}q^{\frac{1}{2}},-|a|^{-1}q^{\frac{1}{2}}\big)_{\infty},
\\
 \bigg|{\prod_{l=1}^{s+1} \big(\epsilon_{r}^2 -q^{s+l} \big)}\bigg | \ge \prod_{l=1}^{s+1}
 \big(q^{s+l}-|\epsilon_0|^2q^{2r}\big)\ge q^{s(s+1)+\frac{(s+1)(s+2)}{2}} \big(|\epsilon_0|^2q\big)_{\infty}.
\end{gather*}
These inequalities imply~\eqref{eq5.8}.
\end{proof}

\begin{Lemma}\label{lemma5.3} 
If $\kappa_i=0$ for some $i$, then\medskip

Case {\rm I:}
\begin{gather*}
\lim_{r\to \infty}\big({-}\epsilon_r^{-1} q, -\epsilon_r^{-1} aq\big)_{\infty}^n \,\epsilon_r^{n(n-1)k+n_1(n_1+1)}
\int_{C_{r,\kappa_1}\times \cdots \times C_{r,\kappa_n}}
 \chi_1(x) \chi_1(x) ^{\dagger} \Delta(x) \frac{{\rm d}x}{x}= 0,
\end{gather*}

Case {\rm II:}
\begin{gather*}
\lim_{r\to \infty}\big({-}\epsilon_r^{-1} q, -\epsilon_r^{-1} aq\big)_{\infty}^n \epsilon_r^{n(n-1)k+n_1(n_1+1)+n}
\int_{C_{r,\kappa_1}\times \cdots \times C_{r,\kappa_n}}
 \chi_1(x) \chi_1(x) ^{\dagger} \Delta(x) \frac{{\rm d}x}{x}= 0.
 \end{gather*}
\end{Lemma}
\begin{proof}
Put $I_{0}=\{i \mid\kappa_i=0 \}$, $I_{+}=\{i\mid \kappa_i=1\}$, $I_{-}=\{i\mid \kappa_i=-1\}$ and
denote their intersection with $\{ n_0+1,\dots,n \}$ by $ I_{0}'$, $I_{+}'$, $I_{-}'$ respectively.
We write
\begin{equation*}
|I_{0}|= p_0,\qquad
|I_{+}|= p_{+}, \qquad
|I_{-}|= p_{-}, \qquad
|I_{0}'|= p_0', \qquad
|I_{+}'|= p_{+}', \qquad
|I_{-}'|= p_{-}',
\end{equation*}
so that $p_0+p_{+}+p_{-}=n$, $p_0'+ p_{+}'+p_{-}'=n_1$. We assume that $p_0\ge 1$.

Let $x=(x_i)$ be $|x_i|=1$, $i\in I_0$ and $x_i=x_{r,s_i}^{\kappa_i}$, $i\in I_{+}\cup I_{-}$. Note first that
\begin{equation} \label{eq5.10} 
 \big({-}\epsilon_r^{-1}q,-\epsilon_r^{-1}aq\big)_{\infty}| A(x_i)|\le K |\epsilon_r|
\end{equation}
when $|x_i|=1$.
This follows at once from estimate
\begin{equation*}
| A(x_i)|\le K\frac{1}{\big|\big(\epsilon_r^{-1}q^{\frac{1}{2}}x_i, \epsilon_r^{-1}q^{\frac{1}{2}}x_i^{-1}\big)_{\infty}\big|}
\end{equation*}
and equalities
\begin{align*}
\big(\epsilon_r^{-1}q^{\frac{1}{2}}x_i^{\pm 1}\big)_{\infty} & = \big(\epsilon_r^{-1}q^{\frac{1}{2}}x_i^{\pm 1}\big)_r
\, \big(\epsilon_0^{-1}q^{\frac{1}{2}}x_i^{\pm 1}\big)_{\infty}
\\
& = (-1)^r \epsilon_r^{-r} q^{\frac{r}{2}+\frac{r(r-1)}{2}}x_i^{\pm r} \big(\epsilon_0q^{\frac{1}{2}}x_i^{\mp 1}\big)_r \,
 \big(\epsilon_0^{-1}q^{\frac{1}{2}}x_i^{\pm 1}\big)_{\infty}.
 \end{align*}
Hence by~\eqref{eq5.8} and~\eqref{eq5.10} we see
\begin{equation} \label{eq5.11} 
 \bigg|\big(\epsilon_r^{-1} q, -\epsilon_r^{-1} aq\big)_{\infty}^n
\prod_{i\in I_0} A(x_i)
 \!\!\!\! \prod_{i\in I_{+}\cup I_{-}}\!\!\operatorname{Res}_{x_i=x_{r,s_i}^{\kappa_i}} A(x_i)\bigg|
\le K\, \bigg\vert\frac{a}{q}\bigg\vert^{\sum_{i\in I_{+}\cup I_{-}}
s_i } |\epsilon_r|^{n+p_+ +p_-}.
\end{equation}
The following inequalities are also immediate:
\begin{equation} \label{eq5.12} 
\big| \epsilon_r^{(n(n-1)-p_0(p_0-1))(k-1)}D(x) \big| \le K, \qquad
\big| \epsilon_{r}^{n(n-1)-p_0(p_0-1)+n_1(n_1-1)-p_0'(p_0'-1)}h(x) \big| \le K.
\end{equation}

The factor $g(x)$ needs precise analysis. In \emph{Case} I: $a_3=\epsilon_rbq^{\frac{1}{2}}$, $a_4=-q^{\frac{1}{2}}$
\big(resp.\ \emph{Case} II: $a_3=\epsilon_r^{-1}q^{\frac{1}{2}}$, $a_4=-q^{\frac{1}{2}}$\big), we have
\begin{gather*}
\big|(x_i-a_3)(x_{i}-a_4)\big(x_{i}^{-2}-1\big)\big| \le K \quad \big(\text{resp.} \ K |\epsilon_r|^{-1} \big),
\\
\big| (x_{r,s_i}-a_3)(x_{r,s_i}-a_4)\big(x_{r,s_i}^{-2}-1\big)\big| \le K |\epsilon_r|^{-2} |q|^{2s_i }\quad
\big(\text{resp.} \ K |\epsilon_r|^{-2} |q|^{s_i} \big),
\\
\big| (x_{r,s_i}^{-1}-a_3)\big(x_{r,s_i}^{-1}-a_4\big)\big(x_{r,s_i}^2-1\big)\big| \le K|\epsilon_r |^{-1} |q|^{s_i}\quad
\big(\text{resp.} \ K |\epsilon_r|^{-3} |q|^{2s_i}\big);
\\
\big| (1-a_3x_i)(1-a_4x_i)\big(1-a_3^{-1}x_i^{-1}\big)\big(1-a_4^{-1}x_i^{-1}\big)\big| \le K |\epsilon_r|^{-1} \quad
\big(\text{resp.} \ K |\epsilon_r|^{-1} \big),
\\
\big|(1-a_3x_{r,s_i})(1-a_4x_{r,s_i})\big(1-a_3^{-1}x_{r,s_i}^{-1}\big)\big(1-a_4^{-1}x_{r,s_i}^{-1}\big)\big|
\le K |\epsilon_r|^{-1}
\quad \big(\text{resp.} \ K |\epsilon_r|^{-3} \big),
\\
\big| \big(1\!-a_3x_{r,s_i}^{-1}\big)\big(1\!-a_4x_{r,s_i}^{-1}\big)\big(1\!-a_3^{-1}x_{r,s_i}\big)\big(1\!-a_4^{-1}x_{r,s_i}\big)\big|
 \le K|\epsilon_r| ^{-3}|q|^{2s_i} \quad\!\!\! \big(\text{resp.}\, K |\epsilon_r|^{-1} |q|^{2s_i} \big),
\end{gather*}
where $|x_i|=1$. Hence
\begin{align} 
& Case~{\rm I}{:} \qquad
\big|\epsilon^{2p_{+}+p_{-}+n_1+2p_{-}'}g(x)\big| \le K |q|^{2\sum_{i\in I_{+}} s_i+\sum_{i\in I_{-}}s_i
+2\sum_{i\in I_{-}' }s_i}, \label{eq5.13} \\
 & Case~{\rm II}{:} \qquad
\big|\epsilon^{n+p_{+}+2p_{-}+n_1+2p_{+}'}g(x)\big| \le K |q|^{\sum_{i\in I_{+}} s_i+2\sum_{i\in I_{-}}s_i+2\sum_{i\in I_{-}' }s_i} . \label{eq5.14}
 \end{align}
Now combining~\eqref{eq5.11}, \eqref{eq5.12}, \eqref{eq5.13} and~\eqref{eq5.14} yields\medskip

\emph{Case} I:
\begin{gather*}
\bigg|\big(\epsilon_r^{-1} q, -\epsilon_r^{-1} aq\big)_{\infty}^n \epsilon_r^{n(n-1)k+n_1(n_1+1)}g(x)h(x)D(x) \prod_{i\in I_0} A(x_i)
\prod_{i\in I_{+}\cup I_{-}}\operatorname{Res}_{x_i=x_{r,s_i}^{\kappa_i}} A(x_i)\bigg|
\\ \qquad
{} \le K \epsilon_r^{e_1} \vert{a}\vert^{\sum_{i\in I_{+}\cup I_{-}}s_i } |q|^{\sum_{i\in I_{+}} s_i+2\sum_{i\in I_{-}' }s_i} ,
\end{gather*}

\emph{Case} II:
\begin{gather*}
\bigg|\big(\epsilon_r^{-1} q, -\epsilon_r^{-1} aq\big)_{\infty}^n \,\epsilon_r^{n(n-1)k+n_1(n_1+1)+n}g(x)h(x)
D(x) \prod_{i\in I_0} A(x_i)
 \prod_{i\in I_{+}\cup I_{-}}\operatorname{Res}_{x_i=x_{r,s_i}^{\kappa_i}} A(x_i)\bigg|
 \\ \qquad
{} \le K \epsilon_r^{e_2} \vert{a}\vert^{\sum_{i\in I_{+}\cup I_{-}}s_i } |q|^{\sum_{i\in I_{-}}s_i+2\sum_{i\in I_{-}' }s_i} ,
 \end{gather*}
where
\begin{gather*}
e_1= p_0(p_0-1)k+p_0'(p_0'-1)+ n+n_1-p_{+} -2p_{-}',
\\
e_2 = p_0(p_0-1)k+p_0'(p_0'-1)+ n+n_1-p_{-} -2p_{+}'.
\end{gather*}
Since
\begin{gather*}
n+n_1-p_{+} -2p_{-}' = p_{0}+p_{0}'+p_{+}'+p_{-}-p_{-}' \ge p_0 \ge 1 ,
\\
 n+n_1-p_{-} -2p_{+}' = p_{0}+p_{0}'+p_{-}'+p_{+}-p_{+}' \ge p_0 \ge 1,
\end{gather*}
and the series $\sum_{} \vert{a}\vert^{\sum_{i\in I_{+}\cup I_{-}}s_i } $ is absolutely convergent, we obtain the conclusion of our lem\-ma.
\end{proof}

 \begin{Theorem}\label{theorem5.4} 
 We have\medskip

Case {\rm I:}
\begin{gather*}
\lim_{r\to \infty}\big({-}\epsilon_r^{-1} q, -\epsilon_r^{-1} aq\big)_{\infty}^n \epsilon_r^{n(n-1)k+n_1(n_1+1)}
 \frac{1}{(2\pi {\rm i})^n}\int_{C_r^n}
 \chi_1(x) \chi_1(x) ^{\dagger} \Delta(x) \frac{{\rm d}x}{x}
 \\ \qquad
{} = A_{I} \sum_{s_1,\dots,s_n=0}^{\infty} \,\prod_{i=1}^{n} a^{s_i}q^{((n-1)k +1)s_i} \frac{(q^{s_i+1})_{\infty}}{(bq^{s_i+1})_{\infty}}
\prod_{1\le i< j\le n}\big(q^{s_i-s_j}, q^{-s_i+s_j+1}\big)_k
\\ \qquad\hphantom{=}
 {}\times\prod_{i=n_0+1}^nq^{s_i}\big(1-bq^{s_i+1}\big)
 \prod_{n_0+1\le i<j \le n}\big(t^{-2}q^{s_i}-q^{s_j}\big)\big(q^{s_i}-t^{-2}q^{s_j}\big),
\end{gather*}

Case {\rm II:}
\begin{gather*}
 \lim_{r\to \infty}\big({-}\epsilon_r^{-1} q, -\epsilon_r^{-1} aq\big)_{\infty}^n \epsilon_r^{n(n-1)k+n_1(n_1+1)+n}
 \frac{1}{(2\pi {\rm i})^n} \int_{C_r^n}
 \chi_1(x) \chi_1(x) ^{\dagger} \Delta(x) \frac{{\rm d}x}{x}
 \\ \qquad
{} = A_{II} \sum_{s_1,\dots,s_n=0}^{\infty} \prod_{i=1}^{n} a^{s_i}q^{((n-1)k +1)s_i}
 \frac{(q^{s_i+1})_{\infty}}{\big(bq^{s_i+1}\big)_{\infty}} \prod_{1\le i< j\le n}\big(q^{s_i-s_j+1}, q^{-s_i+s_j}\big)_k
 \\ \qquad\hphantom{=}
 {}\times\prod_{i=n_0+1}^nq^{s_i}\big(1-q^{s_i}\big)
 \prod_{n_0+1\le i<j \le n}\big(t^2q^{s_i}-q^{s_j}\big)\big(q^{s_i}-t^2q^{s_j}\big),
 \end{gather*}
 where
 \begin{gather*}
 A_{I}= (-1)^{\binom{n}{2}k} (q)_{\infty}^{-2n}\,q^{\binom{n_1}{2}+\binom{n}{2}
\frac{k(k+1)}{2}}\,t^{n_1(n_1-1)} b^{-n_1},
\\
 A_{II} =(-1)^{\binom{n}{2}k} (q)_{\infty}^{-2n}\,q^{\frac{n_1(n_1+1)}{2}+n+
\binom{n}{2} \frac{k(k+1)}{2}}\,t^{n_0(n_0-1)+2n_0n_1} .
\end{gather*}
\end{Theorem}

\begin{proof}
The integral in the theorem are
sums of the integrals
over $C_{r,\kappa_1} \times \cdots \times C_{r,\kappa_n}$, $\kappa_i=0,\pm 1$, see~\eqref{eq5.2}.
By Lemma~\ref{lemma5.3} it suffices to consider only the integrals over $C_{r,\kappa_1} \times
\cdots \times C_{r,\kappa_n}$ with $\kappa_i=\pm 1$, $i=1,\dots,n$,
 in which case each integral reduces to the sum of residues of
the poles $ \big(x_{r,s_1}^{\kappa_1},\dots,x_{r,s_n}^{\kappa_n}\big)$, $0\le s_i\le r-1$, $i=1,\dots,n$, see~\eqref{eq5.3}.
We assert that the order with respect to $\epsilon_r$ of the factor $g\big(x_{r,s_1}^{\kappa_1}, \dots, x_{r,s_n}^{\kappa_n} \big)h\big(x_{r,s_1}^{\kappa_1}, \dots, x_{r,s_n}^{\kappa_n} \big)
 D\big(x_{r, s_1}^{\kappa_1},\dots,x_{r,s_n}^{\kappa_n}\big) $ in the right-hand sum of~\eqref{eq5.3}
 is lowest only when
 $\kappa_1=\dots=\kappa_{n_0}=1$, $\kappa_{n_0 +1}= \dots =\kappa_n = -1$ in \emph{Case}~I and
 $ \kappa_1=\dots=\kappa_{n_0}=-1$, $\kappa_{n_0 +1}= \dots =\kappa_n = 1$ in \emph{Case}~II.
To see this, in view of Lemma~\ref{lemma5.1}, one only needs to consider how the order of $g\big( x_{r,s_1}^{\kappa_1},\dots,x_{r,s_n}^{\kappa_n}\big)$ depends on the choice of $\kappa$.
Let us look at the exponents $e_1$, $e_2$ in the proof of Lemma~\ref{lemma5.3}. Since $p_0=p_0'=0$, we see
\begin{equation*}
e_1=p_{+}'+p_{-}-p_{-}' , \qquad e_2=p_{-}'+p_{+}-p_{+}' .
\end{equation*}
Hence $e_1$ (resp.~$e_2$) is smallest when $p_{+}'=p_{-}-p_{-}' =0$ (resp.\ $p_{-}'=p_{+}-p_{+}'=0 $), which is nothing but
our assertion.

Since\vspace{-1mm}
\begin{gather*}
 \big|\epsilon_r^2(x_{r,s}-a_3)(x_{r,s}-a_4)\big(x_{r,s}^{- 2}-1\big) \big| \le K|q|^{2s},
 \\
 \big|\epsilon_r\big(x_{r,s}^{-1}-a_3\big)\big(x_{r,s}^{-1}-a_4\big)\big(x_{r,s}^{2}-1\big) \big| \le K|q|^s,
\\
 \lim_{r\to \infty}\epsilon_r^2(x_{r,s}-a_3)(x_{r,s}-a_4)\big(x_{r,s}^{- 2}-1\big)=-q^{2s+1},
\\
\lim_{r\to \infty}\epsilon_r\big(x_{r,s}^{-1}-a_3\big)\big(x_{r,s}^{-1}-a_4\big)\big(x_{r,s}^{2}-1\big)= \big(1-bq^{s+1}\big)q^{s+1},
\end{gather*}
and\vspace{-1mm}
\begin{gather*}
 \big| \epsilon_r^3\big(1-a_3x_{r,s}^{-1}\big)\big(1-a_4x_{r,s}^{-1}\big)\big(1-a_3^{-1}x_{r,s}\big)\big(1-a_4^{-1}x_{r,s}\big) \big| \le K|q|^{2s},
\\
\lim_{r\to \infty}\epsilon_r^3 \big(1-a_3x_{r,s}^{-1}\big)\big(1-a_4x_{r,s}^{-1}\big)\big(1-a_3^{-1}x_{r,s}\big)\big(1-a_4^{-1}x_{r,s}\big)=-b^{-1}q^{2s} ,
\end{gather*}
we see\vspace{-1mm}
\begin{gather*}
 \big|\epsilon_r^{2n+2n_1}g\big(x_{r, s_1},\dots,x_{r,s_{n_0}},x_{r,s_{n_0+1}}^{-1},\dots,x_{r,s_n}^{-1}\big)\big|\le K\prod_{i=1}^n |q|^{2s_i}
 \prod_{i=n_0+1}^n |q|^{s_i},
\\
 \lim_{r\to \infty}\epsilon_r^{2n+2n_1}g\big(x_{r, s_1},\dots,x_{r,s_{n_0}},x_{r,s_{n_0+1}}^{-1},\dots,x_{r,s_n}^{-1}\big)
 \!=(-1)^nq^n b^{-n_1}\!\prod_{i=1}^n \! q^{2s_i}\!\!\!\!\prod_{i=n_0+1}^{n}\!\!\!\big(1\!-bq^{s_i+1}\big)q^{s_i}.
\end{gather*}

Similarly, in \emph{Case} II, since\vspace{-1mm}
\begin{gather*}
\bigl|\epsilon_r^2(x_{r,s}-a_3)(x_{r,s}-a_4)\big(x_{r,s}^{- 2}-1\big)\bigr| \le K|q|^s,
\\
\bigl|\epsilon_r^3\big(x_{r,s}^{-1}-a_3\big)\big(x_{r,s}^{-1}-a_4\big)\big(x_{r,s}^{2}-1\big)\bigr| \le K |q|^{2s}
\\
\lim_{r\to \infty}\epsilon_r^2(x_{r,s}-a_3)(x_{r,s}-a_4)\big(x_{r,s}^{- 2}-1\big)=(1-q^s)q^{s+1},
\\
\lim_{r\to \infty}\epsilon_r^3\big(x_{r,s}^{-1}-a_3\big)\big(x_{r,s}^{-1}-a_4\big)\big(x_{r,s}^{2}-1\big)=-q^{2s+2},
\end{gather*}
and\vspace{-1mm}
\begin{gather*}
 \bigl|\epsilon_r^3(1-a_3x_{r,s})(1-a_4x_{r,s})\big(1-a_3^{-1}x_{r,s}^{-1}\big)\big(1-a_4^{-1}x_{r,s}^{-1} \big)\bigr| \le K |q|^{2s},
\\
 \lim_{r\to \infty}\epsilon_r^3(1-a_3x_{r,s})(1-a_4x_{r,s})\big(1-a_3^{-1}x_{r,s}^{-1}\big)\big(1-a_4^{-1}x_{r,s}^{-1} \big)=-q^{2s+2},
\end{gather*}
it follows that\vspace{-1mm}
\begin{gather*}
 \big|\epsilon_r^{3n+2n_1}g\big(x_{r, s_1}^{-1},\dots,x_{r,s_{n_0}}^{-1},x_{r,s_{n_0+1}},\dots,x_{r,s_n}\big)\big|\le K\prod_{i=1}^n q^{2s_i}
 \prod_{i=n_0+1}^n q^{s_i},
 \\
 \lim_{r\to \infty}\epsilon_r^{3n+2n_1}g\big(x_{r, s_1}^{-1},\dots,x_{r,s_{n_0}}^{-1},x_{r,s_{n_0+1}},\dots,x_{r,s_n}\big)
 = (-1)^nq^{2n+n_1}\prod_{i=1}^n\! q^{2s_i}\prod_{i=n_0+1}^{n}\!\big(1-q^{s_i}\big)q^{s_i}.
\end{gather*}

Combining Lemmas~\ref{lemma5.1} and~\ref{lemma5.2}, we have thus obtained estimates of the
sums of residues and their limits as $r\to \infty$. One sees
\begin{gather*}
\bigg| \big({-}\epsilon_r^{-1} q, -\epsilon_r^{-1} aq\big)_{\infty}^n \,\epsilon_r^{n(n-1)k+n_1(n_1+1)}g\big(x_{r, s_1},
\dots,x_{r,s_{n_0}},x_{r,s_{n_0+1}}^{-1},\dots,x_{r,s_n}^{-1}\big)
\\ \qquad
{}\times h\big(x_{r, s_1},\dots,x_{r,s_{n_0}},x_{r,s_{n_0+1}}^{-1},\dots,x_{r,s_n}^{-1}\big)
D(x_{r, s_1},\dots,x_{r,s_n}) \prod_{i=1}^{n} \operatorname{Res}_{x_i=x_{r,s_i}}A_i(x)\bigg| \\ \qquad
{}\le K \prod_{i=1}^n(a q)^{s_i} \prod_{i=n_0+1}^n q^{s_i},
 \\
\bigg| \big({-}\epsilon_r^{-1} q, -\epsilon_r^{-1} aq\big)_{\infty}^n \,\epsilon_r^{n(n-1)k+n+n_1(n_1+1)}g\big(x_{r, s_1}^{-1},
\dots,x_{r,s_{n_0}}^{-1},x_{r,s_{n_0+1}},\dots,x_{r,s_n} \big)
\\ \qquad
{}\times h\big(x_{r, s_1}^{-1},\dots,x_{r,s_{n_0}}^{-1},x_{r,s_{n_0+1}},\dots,x_{r,s_n} \big)
D(x_{r, s_1},\dots,x_{r,s_n}) \prod_{i=1}^{n} \operatorname{Res}_{x_i=x_{r,s_i}}A_i(x)\bigg|
\\ \qquad
{}\le K \prod_{i=1}^n(a q)^{s_i} \prod_{i=n_0+1}^n q^{s_i},
\end{gather*}
for the \emph{Cases} I and II respectively. The series $ \sum_{s_1,\dots,s_n=0}^{\infty} \prod_{i=1}^n(a q)^{s_i}
 \prod_{i=n_0+1}^n q^{s_i} $ is absolutely convergent as $|a|<1$. Therefore, in virtue of the dominated convergence
 theorem, these estimates assure that the residue sums converge as $ r \to \infty$ and it is easily seen that the limits
 are as described in the theorem. This completes the proof of Theorem~\ref{theorem5.4}.
 \end{proof}

\begin{Remark} 
In the case $a_3=\epsilon^{-1}q^{\frac{1}{2}}$, $a_4=\epsilon bq^{\frac{1}{2}} $ or
$a_3=\epsilon bq^{\frac{1}{2}}$, $a_4=\epsilon^{-1}q^{\frac{1}{2}} $,
 on the contrary to our \emph{Case} I and \emph{Case} II, it is not certain whether Lemma~\ref{lemma5.3} holds or not.
 One needs more precise analysis on $A(x)$ when $|x|=1$.
\end{Remark}

\section{Limit of the quadratic norm formula: Part 2} \label{section6}

In this section we calculate the limit of the right-hand side of the norm formula~\eqref{eq4.16}, that is
\begin{gather*}
\lim_{r\to \infty}\big({-}\epsilon^{-1}_r q, -\epsilon^{-1}_r a q\big)_{\infty}^n \epsilon_r^{n(n-1)k+n_1(n_1+1)}
\frac{1}{c_{\rho_1 \rho_1} } \langle E_{\rho_1}, E_{\rho_1}\rangle\qquad
\text{in \emph{Case} I},
\\
\lim_{r\to \infty}\big({-}\epsilon^{-1}_r q, -\epsilon^{-1}_r a q\big)_{\infty}^n \epsilon_r^{n(n-1)k+n_1(n_1+1)+n}
\frac{1}{c_{\rho_1 \rho_1} } \langle E_{\rho_1}, E_{\rho_1}\rangle\qquad
\text{in \emph{Case} II}.
\end{gather*}

Recall that we have written the explicit formula for $ \frac{1}{c_{\rho_1 \rho_1} }
\langle E_{\rho_1}, E_{\rho_1}\rangle $ as (see~\eqref{eq4.18})
 \begin{equation*}
\frac{1}{c_{\rho_1 \rho_1} } \langle E_{\rho_1}, E_{\rho_1}\rangle =F(a_1,a_2, a_3,a_4) G(a_1,a_2,a_3,a_4)
\end{equation*}
in which the factor $G(a_1,a_2,a_3,a_4)$ is invariant under the permutation of $a_1$, $a_2$, $a_3$, $a_4$.
In view of~\eqref{eq4.19} the limit of $F(a_1,a_2, a_3,a_4)$ is readily observed to be
\begin{equation} \label{eq6.1} 
 \lim_{r\to \infty} F(a_1,a_2, a_3,a_4)
 = (ab)^{-n_1} q^{-\frac{n_1(n_1-1)}{2}-2n_1}(t^2)^{-n_0n_1-\frac{n_1(n_1-1)}{2}},
\end{equation}
in \emph{Case} I and
\begin{equation} \label{eq6.2} 
\begin{split}
\lim_{r\to \infty} \epsilon_r^{n}F(a_1a_2, a_3a_4)=q^{\frac{n_1(n_1+1)}{2}+n}(t^2)^{n_1(n-1)+\frac{n_0(n_0-1)}{2}}
 \end{split}
\end{equation}
in \emph{Case} II. The limit of $G(a_1,a_2,a_3,a_4)$ is given by the following proposition.

\begin{Proposition}
We have
\begin{gather}
\lim_{r\to \infty}\big({-}\epsilon^{-1}_r q, -\epsilon^{-1}_r a q\big)_{\infty}^n \epsilon_r^{n(n-1)k+n_1(n_1+1)}G(a_1,a_2,a_3,a_4) \nonumber
\\ \qquad
{}={q^{B}}\frac{\big(qt^2\big)_{\infty}^n}{(q)_{\infty}^n} \prod_{i=1}^{n_0} \frac{1}{\big(qt^{2i}\big)_{\infty}}
\prod_{i=1}^{n_1} \frac{1-t^{2i}}{\big(1-t^2\big)\big(q^i t^{2(n_0+i)}\big)_{\infty}}
\prod_{i=1}^{n_0}
\frac{\big(abq^2t^{2(n_0+i-2)}\big)_{\infty}}{\big(aqt^{2(i-1)},bqt^{2(i-1)}\big)_{\infty}} \nonumber
\\ \qquad\hphantom{=}
{}\times \frac{\prod_{i=1}^{2n_1} \big(abq^{i+2} t^{2(2n_0+i-2)}\big)_{\infty}}{\prod_{i=1}^{n_1}
\big(abq^{i+2}t^{2(n_0+i-2)}, aq^{i+1}t^{2(n_0+i-1)},bq^{i+1}t^{2(n_0+i-1)}\big)_{\infty} } \nonumber
\\ \qquad
{} =\frac{q^{B}}{(1-q)^n(q)_{\infty}^{2n}} \prod_{i=1}^{n_1} \frac{1-q^{ik}}{1-q^k} \frac{1-q^{k+1}}{1-q^{i(k+1)}}\nonumber
\\ \qquad\hphantom{=}
{}\times \prod_{i=1}^{n_0} \frac{\Gamma_q(1+ik)\Gamma_q(\alpha +1+(i-1)k)
\Gamma_q(\beta+1+(i-1)k)} {\Gamma_q(1+k)\Gamma_q(\alpha+\beta+2+(n_0+i-2)k)} \nonumber
\\ \qquad\hphantom{=}
{}\times\prod_{i=1}^{n_1}\frac{[(k+1)i]_q \,\Gamma_q(i+(n_0+i) k)}{\Gamma_q(2+k)}
\Gamma_q(\alpha +i+1+(n_0+i-1)k)\nonumber
\\ \qquad\hphantom{=}
{}\times\Gamma_q(\beta +i+1+(n_0+i-1)k)\Gamma_q(\alpha+\beta+i+2+(n_0+i-2)k)\nonumber
\\ \qquad\hphantom{=}
{}\times\prod_{i=1}^{2n_1}\Gamma_q(\alpha+\beta +i+2+(2n_0+i-2)k )^{-1},
\label{eq6.3}
\end{gather}
where
\begin{align*}
B={}& \frac{n_1(n_1\!+1)(n_1\!+2)}{3}\!+ \frac{n_1(n_1\!+1)}{2} \alpha
+\! \bigg(\!\binom{n}{2}\!-n_1(n_1\!-1)\!+\frac{n_1(n_1\!+1)(3n\!-n_1\!-2)}{3}\bigg)k
\\
&+ \binom{n}{2}\alpha k +\frac{n(n-1)(2n-1)}{6}k^2.
\end{align*}

\end{Proposition}
\begin{proof}
Since
$G(a_1,a_2,a_3,a_4)$ is invariant under the permutation of $a_1$, $a_2$, $a_3$, $a_4$,
we may assume $ (a_1,a_2,a_3,a_4) = \big(\epsilon^{-1}q^{\frac{1}{2}},-aq^{\frac{1}{2}},
\epsilon bq^{\frac{1}{2}}, -q^{\frac{1}{2}}\big)$. In~\eqref{eq4.20} note first that
\begin{equation*}
 \prod_{i=1}^{n_0}\frac{1}{ \big(a_ra_st^{2(i-1)}\big)_{\infty}}
\prod_{i=1}^{n_1}\frac{1}{ \big(a_ra_sq^{i}t^{2(n_0+i-1)}\big)_{\infty}}
=\prod_{i=1}^n \frac{1}{\big(a_ra_st^{2(i-1)}\big)_{\infty}}
\prod_{i=1}^{n_1}\big(a_ra_st^{2(n_0+i-1)}\big)_{i}.
\end{equation*}
In view of $ a_1a_2= -\epsilon_r^{-1}aq$, $a_1a_4 = -\epsilon_r^{-1}q $, one has
\begin{gather*}
\lim_{r\to \infty} \epsilon_r^{n_1(n_1+1)}\prod_{i=1}^{n_1}\big(a_1a_2t^{2(n_0+i-1)}\big)_{i}\,
\big(a_1a_4t^{2(n_0+i-1)}\big)_{i}
\\ \qquad
{}= a^{\frac{n_1(n_1+1)}{2}}
q^{\frac{n_1(n_1+1)(n_1+2)}{3}}(t^2)^{\frac{n_1(n_1+1)(3n-n_1-2)}{3}} ,
\\
\lim_{r \to \infty} \epsilon_r^{n(n-1)k} (-\epsilon_r^{-1}q,-\epsilon_r^{-1}aq)_{\infty}^n
\prod_{i=1}^{n}\frac{1}{\big(a_1a_2t^{2(i-1)},a_1a_4t^{2(i-1)}\big)_{\infty}}
= a^{\binom{n}{2}k}q^{\frac{n(n-1)(2n-1)}{6}k^2+\binom{n}{2}k},
\end{gather*}
 from which the expression for $B$ follows.
The limit of the other factors of $G(a_1,a_2,a_3,a_4)$ is readily observed and one obtains
the first equality of~\eqref{eq6.3}. Rewriting this in terms of $\Gamma_q$ is also straightforward.
Note that the factor $ \frac{1-q^{k+1}}{1-q^{i(k+1)}} $
comes from the shift of $\Gamma_q(1+k)$ to $\Gamma_q(2+k)$:
\begin{equation*}
\frac{1}{\Gamma_q(1+k)}\frac{1-q^{i(k+1)}}{1-q^{k+1}}
=\frac{1}{\Gamma_q(2+k)}[i(k+1)]_q, \qquad i=1,\dots,n_1. \tag*{\qed}
\end{equation*}
 \renewcommand{\qed}{}
\end{proof}
Combining the limit formulas~\eqref{eq6.1}, \eqref{eq6.2} and~\eqref{eq6.3}, we arrive at

\begin{Theorem} \label{theorem6.2} 
We have\medskip

Case {\rm I:}
\begin{gather*}
\lim_{r\to \infty}\big({-}\epsilon^{-1}_r q, -\epsilon^{-1}_r a q\big)_{\infty}^n \epsilon_r^{n(n-1)k+n_1(n_1+1)}
\frac{1}{c_{\rho_1 \rho_1} } \langle E_{\rho_1}, E_{\rho_1}\rangle
\\ \qquad
{} = \frac{q^{B_I}}{(1-q)^n(q)_{\infty}^{2n}}
\prod_{i=1}^{n_1} \frac{1-q^{ik}}{1-q^k} \frac{1-q^{k+1}}{1-q^{i(k+1)}}
\\ \qquad\phantom{=}
{}\times \prod_{i=1}^{n_0} \frac{\Gamma_q(1+ik)\Gamma_q(\alpha +1+(i-1)k)\Gamma_q(\beta+1+(i-1)k)}
 {\Gamma_q(1+k)\Gamma_q(\alpha+\beta+2+(n_0+i-2)k)}
 \\ \qquad\phantom{=}
{}\times \prod_{i=1}^{n_1}\frac{[(k+1)i]_q \,\Gamma_q(i+(n_0+i) k)}{\Gamma_q(2+k)}
 \Gamma_q(\alpha +i+1+(n_0+i-1)k)
 \\[-1ex] \qquad\phantom{=\times \prod_{i=1}^{n_1}}
{}\times \Gamma_q(\beta +i+1+(n_0+i-1)k)\Gamma_q(\alpha+\beta+i+2+(n_0+i-2)k)
\\[-2ex] \qquad\phantom{=}
{}\times \prod_{i=1}^{2n_1}\Gamma_q(\alpha+\beta +i+2+(2n_0+i-2)k )^{-1},
 \end{gather*}

Case {\rm II:}
 \begin{gather*}
\lim_{r\to \infty}\big({-}\epsilon^{-1}_r q, -\epsilon^{-1}_r a q\big)_{\infty}^n
 \epsilon_r^{n(n-1)k+n_1(n_1+1)+n} \frac{1}{c_{\rho_1 \rho_1} }
 \langle E_{\rho_1}, E_{\rho_1}\rangle
 \\ \qquad
{}= \frac{q^{B_{II}}}{(1-q)^n(q)_{\infty}^{2n}} \prod_{i=1}^{n_1}
 \frac{1-q^{ik}}{1-q^k} \frac{1-q^{k+1}}{1-q^{i(k+1)}}
 \\ \qquad\phantom{=}
{}\times \prod_{i=1}^{n_0} \frac{\Gamma_q(1+ik)\Gamma_q(\alpha +1+(i-1)k)\Gamma_q(\beta+1+(i-1)k)}
 {\Gamma_q(1+k)\Gamma_q(\alpha+\beta+2+(n_0+i-2)k)}
 \\ \qquad\phantom{=}
{}\times \prod_{i=1}^{n_1}\frac{[(k+1)i]_q \,\Gamma_q(i+(n_0+i) k)}{\Gamma_q(2+k)}
 \Gamma_q(\alpha +i+1+(n_0+i-1)k)
 \\[-1ex] \qquad\phantom{=\times \prod_{i=1}^{n_1}}
{}\times \Gamma_q(\beta +i+1+(n_0+i-1)k)\Gamma_q(\alpha+\beta+i+2+(n_0+i-2)k)
\\[-2ex] \qquad\phantom{=}
{}\times \prod_{i=1}^{2n_1}\Gamma_q(\alpha+\beta +i+2+(2n_0+i-2)k )^{-1},
 \end{gather*}
 where
 \begin{gather*}
 B_{I} =B-\frac{n_1(n_1+3)}{2}-(\alpha+\beta)n_1-\bigg(n_0 n_1+\frac{n_1(n_1-1)}{2}\bigg)k,
 \\
 B_{II} =B+ \frac{n_1(n_1+1)}{2}+n +\bigg( n_1(n-1)+\frac{n_0(n_0-1)}{2} \bigg) k.
\end{gather*}
\end{Theorem}

\begin{Remark} \label{remark6.3} 
In Sections \ref{section5} and \ref{section6} we have taken the scaling limit of the quadratic norm formula of specific
Koornwinder polynomial to obtain $q$-integral formula.
As stated in introduction, this would be considered as a specific but nonsymmetric case of
Stokman's procedure \cite{St1}, in which he, in particular, deduced the quadratic norm formulas of
the little and big $q$-Jacobi polynomials from that of symmetric Koornwinder polynomials.
Hence it would be desirable to consider the scaling limit of the full nonsymmetric Koornwinder polynomials,
which, in particular, should give {\it nonsymmetric} little $q$-Jacobi polynomials and its quadratic norm formula
(see \cite{St1, SK} for the symmetric case). Moreover partial antisymmetrization of some specific polynomial
of this type might lead to the evaluation of the integral \eqref{eq7.6} in Theorem \ref{theorem7.4} below. Studies on
the scaling limit of the nonsymmetric Koornwinder polynomials are left for future work.
We note that in~\cite{B} (see also \cite{BDF}) particular cases of Theorem \ref{theorem7.4} were affirmed
using partial antisymmetrization of nonsymmetric Macdonald polynomials.
\end{Remark}

 \section[q-Selberg integral]{$\boldsymbol{q}$-Selberg integral} \label{section7}

 We rewrite the infinite series given in Theorem~\ref{theorem5.4} in terms of a $q$-integral.

We shall use the following lemma on antisymmetric functions. This is a special case of Kadell's formula
\cite[formula~(4.7)]{Kad}, and has been appeared also in \cite[p.~1479]{Hab}, \cite[Lemma~3.1]{W2}
and \cite[Lemma 9]{B}. We write $[n]_q != [1]_q\cdots [n]_q$.

 \begin{Lemma} 
 Let $n=n_0+n_1$ with $n_1\ge 2$ and $ f(t_1,\dots,t_n)$ be antisymmetric with respect to
 the variables $t_{n_0+1},\dots, t_n$. Then
 \begin{gather*}
 \int_{[0,1]^n} \prod_{n_0+1\le i< j\le n}(t_i-Qt_j) f(t_1,\dots,t_n) \, {\rm d}_qt_1 \cdots {\rm d}_qt_n
 \\ \qquad
{}= \frac{[n_1]_{Q} !}{n_1!} \int_{[0,1]^n} \prod_{n_0+1\le i< j\le n}(t_i-t_j)
 f(t_1,\dots,t_n) \, {\rm d}_qt_1 \cdots {\rm d}_qt_n.
 \end{gather*}
 \end{Lemma}

This lemma implies
 \begin{gather}
\int_{[0,1]^n} \prod_{n_0+1\le i< j\le n}(t_i-Q_1t_j) f(t_1,\dots,t_n) \, {\rm d}_qt_1 \cdots {\rm d}_qt_n \nonumber
\\ \qquad
{}= \frac{[n_1]_{Q_1} !}{[n_1]_{Q_2} !} \int_{[0,1]^n} \prod_{n_0+1\le i< j\le n}(t_i-Q_2t_j)
 f(t_1,\dots,t_n) \, {\rm d}_qt_1 \cdots {\rm d}_qt_n.\label{eq7.1}
 \end{gather}

\begin{Proposition} \label{proposition7.2} 
We have\medskip

Case {\rm I:}
\begin{gather*}
\sum_{s_1,\dots,s_n=0}^{\infty}\prod_{i=1}^{n} a^{s_i}q^{((n-1)k +1)s_i}
\frac{(q^{s_i+1})_{\infty}}{(bq^{s_i+1})_{\infty}}
\prod_{1\le i< j\le n}\big(q^{s_i-s_j}, q^{-s_i+s_j+1}\big)_k
\\ \qquad
 {}\times \prod_{i=n_0+1}^nq^{s_i}(1-bq^{s_i+1})
 \prod_{n_0+1\le i<j \le n}\big(t^{-2}q^{s_i}-q^{s_j}\big)\big(q^{s_i}-t^{-2}q^{s_j}\big)
\\ \qquad
{} =\frac{(-1)^{\binom{n}{2}k}q^{\binom{n}{2}\frac{k(k-1)}{2}-\binom{n_1}{2}k}}{(1-q)^n}
 \prod_{i=1}^{n_1} \frac{1-q^{ik}}{1-q^k} \frac{1-q^{k+1}}{1-q^{i(k+1)}}
\\ \qquad\hphantom{=}
{}\times \int_{[0,1]^n} \prod_{i=n_0+1}^n t_i\big(1-q^{\beta+1}t_i\big)
 \prod_{n_0+1\le i<j\le n}\big(t_i-q^{-k}t_j\big)\big(t_i-q^{k+1}t_j\big)
 \\ \qquad\hphantom{=\times \int_{[0,1]^n}}
{}\times \prod_{i=1}^n t_i^{\alpha}\frac{(qt_i)_{\infty}}{\big(q^{\beta+1}t_i\big)_{\infty}}
\prod_{1\le i<j\le n}
 t_i^{2k}\bigg(q^{1-k}\frac{t_j}{t_i}\bigg)_{2k} \, {\rm d}_qt_1 \cdots {\rm d}_qt_n,
\end{gather*}

Case {\rm II:}
\begin{gather*}
\sum_{s_1,\dots,s_n=0}^{\infty}\prod_{i=1}^{n} a^{s_i}q^{((n-1)k +1)s_i}
\frac{\big(q^{s_i+1}\big)_{\infty}}{(bq^{s_i+1})_{\infty}} \prod_{1\le i< j\le n}\big(q^{s_i-s_j+1}, q^{-s_i+s_j}\big)_k
 \\ \qquad
{}\times\prod_{i=n_0+1}^nq^{s_i}\big(1-q^{s_i}\big)
 \prod_{n_0+1\le i<j \le n}\big(t^2q^{s_i}-q^{s_j}\big)\big(q^{s_i}-t^2q^{s_j}\big),
 \\ \qquad
{} =\frac{(-1)^{\binom{n}{2}k}q^{\binom{n}{2} \frac{k(k+1)}{2}+\binom{n_1}{2}(k+1)}}{(1-q)^n}
\prod_{i=1}^{n_1} \frac{1-q^{ik}}{1-q^k} \frac{1-q^{k+1}}{1-q^{i(k+1)}}
\\ \qquad\hphantom{=}
{}\times \int_{[0,1]^n} \prod_{i=n_0+1}^n t_i(1-t_i)
 \prod_{n_0+1\le i<j\le n}\big(t_i-q^{-k-1}t_j\big)\big(t_i-q^{k}t_j\big)
 \\ \qquad\hphantom{=\times\int_{[0,1]^n}}
{}\times \prod_{i=1}^n t_i^{\alpha}\frac{(qt_i)_{\infty}}{\big(q^{\beta+1}t_i\big)_{\infty}}
\prod_{1\le i<j\le n} t_i^{2k}\Bigl(q^{-k}\frac{t_j}{t_i}\Bigr)_{2k} \,
 {\rm d}_qt_1 \cdots {\rm d}_qt_n.
 \end{gather*}
\end{Proposition}
\begin{proof}
These are rewrites in which use is made of~\eqref{eq7.1}. For \emph{Case} I, we use~\eqref{eq7.1} with $ Q_1=
q^k$, $Q_2=q^{k+1}$ and for \emph{Case} II, $ Q_1=q^{-k}$, $Q_2=q^{-k-1}$.
\end{proof}
 Combining Theorems~\ref{theorem5.4} and~\ref{theorem6.2} and Proposition~\ref{proposition7.2},
 we arrive at a generalization of the $q$-Selberg integral.
 \begin{Theorem} \label{theorem7.3} 
Let $n=n_0+n_1$ with $n_1\ge 2$. For $\mathrm{Re}\, \alpha>0$, $\mathrm{Re}\, \beta>0$, $k\in \mathbb{Z}_{\ge 0}$,
we have
 \begin{gather}
\int_{[0,1]^n} \prod_{i=n_0+1}^n t_i\big(1-q^{\beta}t_i\big) \prod_{n_0+1\le i<j\le n}\big(t_i-q^{-k}t_j\big)\big(t_i-q^{k+1}t_j\big)\nonumber
 \\ \qquad
{}\times \prod_{i=1}^n t_i^{\alpha-1}\frac{(qt_i)_{\infty}}{\big(q^{\beta}t_i\big)_{\infty}}\prod_{1\le i<j\le n}
t_i^{2k}\Bigl(q^{1-k}\frac{t_j}{t_i}\Bigr)_{2k}\, {\rm d}_qt_1\cdots {\rm d}_qt_n\nonumber
 \\ \qquad
{} =q^{C_{I}}
 \prod_{i=1}^{n_0}\frac{\Gamma_q(1+ik)\Gamma_q(\alpha +(i-1)k)\Gamma_q(\beta+(i-1)k)}
 {\Gamma_q(1+k)\Gamma_q(\alpha+\beta+(n_0+i-2)k)} \nonumber
 \\ \qquad\hphantom{=}
{}\times\prod_{j=1}^{n_1}\frac{[(1+k)j]_q \Gamma_q(j+(n_0+j) k)}{\Gamma_q(2+k)}
 \prod_{j=1}^{n_1}\Gamma_q(\alpha +j+(n_0+j-1)k)\nonumber
 \\ \qquad\hphantom{=}
{}\times \Gamma_q(\beta +j+(n_0+j-1)k)
\Gamma_q(\alpha+\beta+j+(n_0+j-2)k) \nonumber
\\ \qquad\hphantom{=}
\times\prod_{j=1}^{2n_1}\big\{\Gamma_q(\alpha+\beta +j+(2n_0+j-2)k ) \big\}^{-1},\label{eq7.2}
 \end{gather}
 where
 \begin{gather}
 C_{I} = \frac{n_1(n_1-1)(2n_1-1)}{6}+\binom {n_1} 2 \alpha +
\bigg(n_0n_1^2+\frac{n_1(n_1-1)(4n_1-5)}{6} \bigg)k \nonumber
\\ \hphantom{ C_{I} =}
 {}+\binom n 2 \alpha k+ 2\binom n 3 k^2 ,\label{eq7.3}
\end{gather}
and
 \begin{gather}
 \int_{[0,1]^n} \prod_{i=n_0+1}^n t_i(1-t_i)
 \prod_{n_0+1\le i<j\le n}\big(t_i-q^{-k-1}t_j\big)\big(t_i-q^{k}t_j\big)\nonumber
 \\ \qquad
 {}\times \prod_{i=1}^n t_i^{\alpha-1}\frac{(qt_i)_{\infty}}{\big(q^{\beta}t_i\big)_{\infty}}
 \prod_{1\le i<j\le n} t_i^{2k}\Bigl(q^{-k}\frac{t_j}{t_i}\Bigr)_{2k} \, {\rm d}_qt_1 \cdots {\rm d}_qt_n \nonumber
 \\ \qquad
 {} = q^{C_{II}}
 \prod_{i=1}^{n_0}\frac{\Gamma_q(1+ik)\Gamma_q(\alpha +(i-1)k)\Gamma_q(\beta+(i-1)k)}
 {\Gamma_q(1+k)\Gamma_q(\alpha+\beta+(n_0+i-2)k)} \nonumber
 \\ \qquad\hphantom{=}
{}\times\prod_{j=1}^{n_1}\frac{[(1+k)j]_q \Gamma_q(j+(n_0+j) k)}{\Gamma_q(2+k)}
 \prod_{j=1}^{n_1}\Gamma_q(\alpha +j+(n_0+j-1)k)\nonumber
 \\ \qquad\hphantom{=}
{}\times\Gamma_q(\beta +j+(n_0+j-1)k) \Gamma_q(\alpha+\beta+j+(n_0+j-2)k) \nonumber
\\ \qquad\hphantom{=}
{}\times\prod_{j=1}^{2n_1}\big\{\Gamma_q(\alpha+\beta +j+(2n_0+j-2)k ) \big\}^{-1},\label{eq7.4}
 \end{gather}
where
 \begin{gather}
 C_{II}= \frac{n_1(n_1^2+2)}{3}+\frac{n_1(n_1+1)}{2} \alpha +
\bigg(\frac{n_1(n_1+1)(3n-n_1-2)}{3}-\binom {n_1} 2 -\binom n 2 \bigg)k \nonumber
\\ \hphantom{ C_{II}=}
{}+ \binom n 2 \alpha k + 2\binom n 3 k^2.\label{eq7.5}
\end{gather}
 \end{Theorem}

 \begin{proof}
 Let us write $C_{I}$, $C_{II}$ in~\eqref{eq7.3} and~\eqref{eq7.5} as $ C_{I}(\alpha)$, $C_{II}(\alpha) $
 respectively. Then, in view of Theorems~\ref{theorem5.4} and~\ref{theorem6.2} and
 Proposition~\ref{proposition7.2},
 it only remains to check that
 \begin{gather*}
 A_{I}^{-1} \frac{q^{B_I}}{(1-q)^n (q)_{\infty}^{2n}}
 \biggl(\frac{(-1)^{\binom{n}{2}k}q^{\binom{n}{2}\frac{k(k-1)}{2}-\binom{n_1}{2}k}}{(1-q)^n} \biggr)^{-1} = q^{C_{I}(\alpha+1)},
 \\
 A_{II}^{-1} \frac{q^{B_{II}}}{(1-q)^n (q)_{\infty}^{2n}}
 \biggl( \frac{(-1)^{\binom{n}{2}k}q^{\binom{n}{2} \frac{k(k+1)}{2}+\binom{n_1}{2}(k+1)}}{(1-q)^n}\biggr)^{-1} = q^{C_{II}(\alpha+1)}.
 \end{gather*}
This completes the proof of Theorem~\ref{theorem7.3}.
 \end{proof}

 As stated in the introduction, our studies stem from the Baker--Forrester constant term (ex-\nolinebreak )
 conjecture~\eqref{eq1.3}.
We rewrite the conjecture in the form of an integral evaluation. Note first that for
a general Laurent polynomial $f(t_1,\dots,t_n)$
there holds \cite[Proposition 4.1]{BF}:
\begin{gather*}
\biggl(\frac{\Gamma_q(x+y)}{\Gamma _q(x)\Gamma_q(y)} \biggr)^n\int_{[0,1]^n} \prod_{i=1}^n t_i^{x-1}
\frac{(qt_i)_{\infty}}{(q^yt_i)_{\infty}}
\, f(t_1,\dots,t_n)\,{\rm d}_qt_1\cdots {\rm d}_qt_n
\\ \qquad
{} =\biggl(\frac{(q)_a(q)_b}{(q)_{a+b}} \biggr)^n \mbox{CT}_{\{t\}}
\prod_{i=1}^n(t_i)_a\, \bigg(\frac{q}{t_i} \bigg)_b \, f\big(q^{-(b+1)}t_1,\dots,q^{-(b+1)}t_n\big)
\end{gather*}
provided $x=-b$ and $y=a+b+1$. Then put
\begin{equation*}
f=f_{n_0,n}=\prod_{n_0+1\le i<j \le n}\bigg(1-q^k \frac{t_i}{t_j} \bigg)\bigg(1-q^{k+1}
\frac{t_j}{t_i} \bigg)\prod_{1\le i<j\le n}
\bigg(q\frac{t_j}{t_i} \bigg)_k \bigg(\frac{t_i}{t_j} \bigg)_k.
\end{equation*}
In view of the constant term identity~\eqref{eq1.3}, this gives
\begin{gather*}
 \biggl(\frac{\Gamma_q(x+y)}{\Gamma _q(x)\Gamma_q(y)} \biggr)^n\int_{[0,1]^n}
\prod_{i=1}^n t_i^{x-1}\frac{(qt_i)_{\infty}}{(q^yt_i)_{\infty}}
\, f_{n_0,n}(t_1,\dots,t_n)\,{\rm d}_qt_1\cdots {\rm d}_qt_n
\\ \qquad
{} =\biggl(\frac{(q)_a(q)_b}{(q)_{a+b}} \biggr)^n (\mbox{RHS of \eqref{eq1.3}}),
\end{gather*}
where $b=-x $ and $ a=x+y-1 $. One can rewrite this, after some calculations (cf.\ \cite {Kan2}),
to get the reformulation
of the results of Baker--Forrester and K\'{a}rolyi et al.:

\begin{Theorem} [Baker--Forrester \cite{BF}, K\'{a}rolyi et al.\ \cite{Kar}] \label{theorem7.4} 
Let $n=n_0+n_1$ with $n_1\ge 2$. For $\mathrm{Re}\, x>0, \, \mathrm{Re}\,y>0,\, k\in \mathbb{Z}_{\ge 0}$,
we have
\begin{gather}
\int_{[0,1]^n} \prod_{i=1}^{n_0} t_i^{n_1-1} \prod_{n_0+1\le i<j\le n}\big(t_i - q^{-k}t_j\big)\big(t_ i- q^{k+1}t_j\big)
\prod_{i=1}^n t_i^{x-1}\frac{(qt_i)_{\infty}}{(q^{y}t_i)_{\infty}} \nonumber
\\ \qquad
{}\times \prod_{1\le i<j\le n}
t_i^{2k}\Bigl(q^{1-k}\frac{t_j}{t_i}\Bigr)_{2k}\, {\rm d}_qt_1\cdots {\rm d}_qt_n \nonumber
\\ \qquad
{} = q^{D} \prod_{i=1}^{n_0}\frac{\Gamma_q(1+ik)\Gamma_q(x+(n_1+i-1)k+n_1-1)\Gamma_q(y+(i-1)k)}
{\Gamma_q(1+k)\Gamma_q(x+y+(n+i-2)k+n_1-1)} \nonumber
\\ \qquad\hphantom{=}
{}\times \prod_{j=1}^{n_1}\frac{[(1+k)j]_q \Gamma_q(j+(n_0+j)k)}
{\Gamma_q(2+k)} \nonumber
\\ \qquad\hphantom{=}
{}\times \prod_{j=1}^{n_1}\frac{\Gamma_q(x+j-1+(j-1)k)\Gamma_q(y+j-1+(n_0+j-1)k)}
{\Gamma_q(x+y+2n_1-j-1+(2n-j-1)k)},\label{eq7.6}
\end{gather}
where
\begin{gather*}
 D = 2\binom {n_1} 3+ \binom {n_1} 2 x +\frac{n_1-1}{2}\bigg( n(n-1)+\frac{n_1(n_1-5)}{3} \bigg)k + \binom n 2 x k
 \\ \hphantom{ D = }
{} +\frac{1}{2}\biggl( \frac{n(n\!-\!1)(2n\!-\!3)}{2} -n_0 n_1(n\!-\!1)
 -\frac{n_0(n_0\!-\!1)(2n_0\!-\!1)+n_1(n_1\!-\!1)(2n_1\!-\!1)}{6} \biggr) k^2.
 \end{gather*}
\end{Theorem}

\begin{Remark} \qquad
\begin{enumerate}\itemsep=0pt
\item[$1.$] As the integrand of~\eqref{eq7.6} suggests, it might be possible to prove Theorem~\ref{theorem7.4}
in a similar way to the proof of Theorem~\ref{theorem7.3}, if one finds some specific Koornwinder polynomial
whose partial antisymmerization
is the discriminant like $ \prod_{n_0+1\le i<j\le n}(1 - q^{-k}t_j/t_i) $.

\item[$2.$] In \cite{Kan2}, we proved that the evaluation of the integral of~\eqref{eq7.6} follows from a conjecture
for the expansion of certain symmetric polynomial related to the integrand of~\eqref{eq7.6} in terms of
Macdonald polynomials. This conjecture, i.e., certain terms do not appear in the expansion, was
checked for small values of $n_0$ and $n_1$. A similar problem, namely the expansion of the square
of the difference product in terms of Schur or Jack polynomials
or of other orthogonal polynomials has been investigated rather extensively,
see \cite{BBL, DGIL} and references therein.
\end{enumerate}
\end{Remark}
 Next we consider the $q=1$ case. Let us write
 \begin{equation*}
 D_{\alpha,\beta,\gamma}(t)=\prod_{i=1}^n t_i^{\alpha-1}(1-t_i)^{\beta-1}
 \prod_{1\le i<j\le n}
|t_i-t_j|^{2\gamma}.
 \end{equation*}
 From Theorems~\ref{theorem7.3} and~\ref{theorem7.4} one obtains

 \begin{Corollary}
For $\mathrm{Re}\, \alpha>0$, $\mathrm{Re}\,\beta >0$, $\mathrm{Re} \, \gamma \ge 0 $, we have
\begin{gather}
 \int_{[0,1]^n}\prod_{i=n_0+1}^n t_i(1-t_i)\prod_{n_0+1\le i<j\le n}(t_i-t_j)^2 \,
 D_{\alpha,\beta,\gamma}(t) \, {\rm d}t_1\cdots {\rm d}t_n \nonumber
 \\ \qquad
{}= \prod_{i=1}^{n_0}\frac{\Gamma(1+i\gamma)\Gamma(\alpha +(i-1)\gamma)\Gamma(\beta+(i-1)\gamma)}
 {\Gamma(1+\gamma)\Gamma(\alpha+\beta+(n_0+i-2)\gamma)}
 \prod_{j=1}^{n_1}\frac{(1+\gamma)j\,\Gamma(j+(n_0+j) \gamma)}{\Gamma(2+\gamma)} \nonumber
 \\ \qquad\hphantom{=}
{}\times\prod_{j=1}^{n_1} \Gamma(\alpha\! +j\!+(n_0\!+j\!-1)\gamma)
 \Gamma(\beta\! +j\!+(n_0+j\!-1)\gamma) \Gamma(\alpha+\beta+j +(n_0+j\!-2)\gamma) \nonumber
 \\ \qquad\hphantom{=}
{}\times\prod_{j=1}^{2n_1}
 \{\Gamma(\alpha+\beta +j+(2n_0+j-2)\gamma ) \}^{-1},\label{eq7.7}
 \\
 \int_{[0,1]^n} \prod_{i=1}^{n_0} t_i^{n_1-1}\prod_{n_0+1\le i<j\le n}(t_i-t_j)^2 \,
 D_{\alpha,\beta,\gamma}(t) \, {\rm d}t_1\cdots {\rm d}t_n \nonumber
 \\ \qquad
{}= \prod_{i=1}^{n_0}
\frac{\Gamma(1+i\gamma)\Gamma(\alpha+(n_1+i-1)\gamma+n_1-1)\Gamma(\beta+(i-1)\gamma)
}{\Gamma(1+\gamma)\Gamma(\alpha+\beta+(n+i-2)\gamma+n_1-1)}\label{eq7.8}
\\ \qquad\hphantom{=}
{}\times \prod_{j=1}^{n_1}
\frac{(1\!+\gamma)j\,\Gamma(j\!+(n_0\!+j) \gamma)
\Gamma(\alpha\!+j\!-1\!+(j\!-1)\gamma)\Gamma(\beta+j-1+(n_0+j\!-1)\gamma)}
{\Gamma(2+\gamma)\Gamma(\alpha+\beta+2n_1-j-1+(2n-j\!-1)\gamma)}. \nonumber
 \end{gather}
 \end{Corollary}

\begin{proof}
 The case that $\gamma$ is a nonnegative integer follows from Theorems~\ref{theorem7.3} and~\ref{theorem7.4} by taking the limit $q\to 1$.
 For complex $\gamma$ with $\mathrm{Re} \,\gamma \ge 0$, one can apply a weak form of
 Carlson's theorem: Suppose $f(z)$ be bounded and holomorphic in $\mathrm{Re} \, z \ge 0$
 and $f(z)=0$ for $z\in \mathbb{Z}_{\ge 0}$.
 Then $f(z)=0$ identically. For fixed $\alpha$ and $\beta$, the integrals on
 the left sides of~\eqref{eq7.7} and~\eqref{eq7.8} are clearly bounded and
 holomorphic in $\mathrm{Re} \, \gamma \ge 0$. The products on the right side can be verified to
 be bounded in $\mathrm{Re}\, \gamma \ge 0$ by Sirling's formula.
 \end{proof}

As a limit case of these fomulae, we obtain a generalization of Mehta integral \cite{FW,MD}.
 Let~${\rm d}\mu$ be the standard Gaussian measure on $\mathbb R^n$:
 \begin{equation*}
 {\rm d}\mu=(2\pi)^{-\frac{n}{2}}\prod_{i=1}^{n}{\rm e}^{-\frac{x_i^2}{2}}{\rm d}x_1\cdots {\rm d}x_n.
 \end{equation*}
 Then
\begin{Corollary} \label{corollary7.7} 
 For $\mathrm{Re}\,\gamma \ge 0 $, we have
\begin{gather} 
 \int_{\mathbb R^n} \prod_{n_0+1\le i<j \le n}(x_i-x_j)^2 \prod_{1\le i<j \le n} |x_i-x_j|^{2\gamma} {\rm d}\mu\nonumber
 \\ \qquad
{} = \prod_{i=1}^{n_0}\frac{\Gamma(1+i\gamma)}{\Gamma(1+\gamma)}
\prod_{j=1}^{n_1}\frac{j(1+\gamma)\Gamma(j+(n_0+j)\gamma)}{\Gamma(2+\gamma)}.\label{eq7.9}
\end{gather}
 \end{Corollary}

 \begin{proof}
 In~\eqref{eq7.7}, put $\alpha=\beta$, $t_i=\big(1+x_i/\sqrt{2\alpha}\big)/2$ and let $\alpha\to \infty$, and apply Stirling's for\-mu\-la.
 \end{proof}

 \begin{Remark} 
\eqref{eq7.8} leads to the same formula by the same way.
 \end{Remark}

We have also
 \begin{Corollary} 
 For $\mathrm{Re}\, \alpha>0$, $\mathrm{Re} \, \gamma \ge 0 $, we have 
\begin{gather}
 \int_{\mathbb R_{\ge 0}^n}\prod_{j=n_0+1}^n x_j\prod_{n_0+1\le i<j \le n}(x_i-x_j)^2
\prod_{1\le i<j \le n} |x_i-x_j|^{2\gamma}\prod _{i=1}^{n}x_i^{\alpha-1}{\rm e}^{-x_i}{\rm d}x \notag
\\ \qquad
{}= \prod_{i=1}^{n_0}
\frac{\Gamma(1+i\gamma)\Gamma(\alpha +(i-1)\gamma)}{\Gamma(1+\gamma)} \notag
\\ \qquad\hphantom{=}
{}\times\prod_{j=1}^{n_1}\frac{j(1+\gamma)\Gamma(j+(n_0+j)\gamma)\Gamma(\alpha +j+(n_0+j-1)
\gamma)}{\Gamma(2+\gamma)}, \label{eq7.10}
\\
 \int_{\mathbb R_{\ge 0}^n}\prod_{n_0+1\le i<j \le n}(x_i-x_j)^2 \prod_{1\le i<j \le n} |x_i-x_j|^{2\gamma}
\prod _{i=1}^{n}x_i^{\alpha-1}{\rm e}^{-x_i} {\rm d}x \notag
\\ \qquad
 {}= \prod_{i=1}^{n_0}\frac{\Gamma(1+i\gamma)\Gamma(\alpha +(i-1)\gamma)}{\Gamma(1+\gamma)}\notag
\\ \qquad\hphantom{=}
{}\times\prod_{j=1}^{n_1}
\frac{j(1+\gamma)\Gamma(j+(n_0+j)\gamma)\Gamma(\alpha +j-1+(n_0+j-1)\gamma)}{\Gamma(2+\gamma)}, \label{eq7.11}
\end{gather}
where ${\rm d}x={\rm d}x_1\cdots {\rm d}x_n$.
\end{Corollary}
\begin{proof}
 For~\eqref{eq7.10}, put $t_i=x_i/\beta$ and let $\beta\to \infty$ in~\eqref{eq7.7}, and apply
 Stirling's formula. For~\eqref{eq7.11}, change variables
 $t_i$ with $1-t_i$ and exchange $\alpha$ and $\beta$ in~\eqref{eq7.8}. Then put $t_i=x_i/\beta$ and
 take the limit $\beta\to \infty$.
\end{proof}

These formulas give, as in \cite{AR} (cf.\ \cite[Theorem 8.3.3]{AAR}), integral formulas on the $n$-simplex.
Let $\Delta_n=\{ (x_1,\dots,x_n)\mid x_i\ge 0,\, i=1,\dots, n,
\, x_1+\cdots+x_n\le 1 \}$. We denote the right-hand sides of~\eqref{eq7.10} and~\eqref{eq7.11} by $\Gamma_1$, $\Gamma_2$ respectively. Then

\begin{Proposition}
 For $\mathrm{Re}\, \alpha>0$, $\mathrm{Re} \, \gamma \ge 0$, we have 
\begin{gather}
 \int_{\Delta_n}\,\, \prod_{j=n_0+1}^n x_j\prod _{i=1}^{n}x_i^{\alpha-1}\,
\bigg(1-\sum_{i=1}^{n}x_i\bigg)^{\beta-1}
\prod_{n_0+1\le i<j \le n}(x_i-x_j)^2 \prod_{1\le i<j \le n} |x_i-x_j|^{2\gamma}{\rm d}x \notag
\\ \qquad
{}= \frac{\Gamma(\beta)}{\Gamma(\beta+n\alpha+n_1^2+n(n-1)\gamma)} \Gamma_1, \label{eq7.12} \\
\int_{\Delta_n}\,\,\prod _{i=1}^{n}x_i^{\alpha-1}\, \bigg(1-\sum_{i=1}^{n}x_i\bigg)^{\beta-1}
\prod_{n_0+1\le i<j \le n}(x_i-x_j)^2 \prod_{1\le i<j \le n} |x_i-x_j|^{2\gamma}{\rm d}x \notag
\\ \qquad
{} = \frac{\Gamma(\beta)}{\Gamma(\beta+n\alpha+n_1(n_1-1)+n(n-1)\gamma)} \Gamma_2. \label{eq7.13}
\end{gather}
\end{Proposition}

\begin{proof}
The proof is almost the same as the one given in \cite{AR} which we briefly reprise. Put
\begin{equation*}
L(\lambda) =\int_{\mathbb R_{\ge 0}^n}\prod_{j=n_0+1}^n x_j\prod_{n_0+1\le i<j \le n}(x_i-x_j)^2
\prod_{1\le i<j \le n} |x_i-x_j|^{2\gamma}\prod _{i=1}^{n}x_i^{\alpha-1}{\rm e}^{-\lambda x_i}{\rm d}x,
\end{equation*}
so that by a change of variables one has
\begin{equation*}
L(\lambda) =\lambda^{-(n\alpha + n_1^2+n(n-1)\gamma)} L(1).
\end{equation*}
Multiplying both sides by $\lambda^{\beta+n\alpha+n_1^2+n(n-1)\gamma-1}{\rm e}^{-\lambda} $
and integrating over $ [0,\infty) $ give
\begin{gather*}
 \int_{\mathbb R_{\ge 0}^n}\frac{\prod_{j=n_0+1}^n x_j\prod_{n_0+1\le i<j \le n}(x_i-x_j)^2
\prod_{1\le i<j \le n}
 |x_i-x_j|^{2\gamma}}{\bigl( 1+ \sum_{i=1}^n x_i\bigr)^{\beta+n\alpha+n_1^2+n(n-1)\gamma}}
\prod _{i=1}^{n}x_i^{\alpha-1}{\rm d}x
\\ \qquad
 {}=\frac{\Gamma(\beta)}{\Gamma(\beta+n\alpha+n_1^2+n(n-1)\gamma)} L(1).
\end{gather*}
Then by the change of the variables
\begin{equation*}
x_i=t_i\bigg( 1-\sum_{i=1}^{n} t_i \bigg)^{-1}
\end{equation*}
we have
\begin{gather*}
\int_{\Delta_n} \prod_{j=n_0+1}^n t_j\prod _{i=1}^{n}
t_i^{\alpha-1}\bigg(1-\sum_{i=1}^{n}t_i\bigg)^{\beta+n}
\prod_{n_0+1\le i<j \le n}(t_i-t_j)^2
\prod_{1\le i<j \le n} |t_i-t_j|^{2\gamma}
 |J| {\rm d}x
 \\ \qquad
{} = \frac{\Gamma(\beta)}{\Gamma(\beta+n\alpha+n_1^2+n(n-1)\gamma)} L(1),
\end{gather*}
where $J$ is the Jacobian. One can readily verify that
\begin{equation*}
J=\bigg( 1-\sum_{i=1}^{n} t_i \bigg)^{-n-1}.
\end{equation*}
This completes the proof of~\eqref{eq7.12}. The case~\eqref{eq7.13} is similar.
\end{proof}

By setting $x_i=y_i^2/2$ in~\eqref{eq7.10} and~\eqref{eq7.11} one also gets
\begin{gather*}
 \int_{\mathbb R^n}\!\prod_{i=n_0+1}^n \!\!x_i^2 \!\!\!\prod_{n_0+1\le i<j \le n}\!\!\big(x_i^2-x_j^2\big)^2\!\!\!
\prod_{1\le i<j \le n}\!\! \big|x_i^2-x_j^2\big|^{2\gamma}
\prod _{i=1}^{n}|x_i|^{2\alpha-1}{\rm e}^{-\frac{x_i^2}{2}}{\rm d}x
 = 2^{n_1^2+n\alpha + n(n-1)\gamma} \Gamma_1,
 \\
 \int_{\mathbb R^n}\! \prod_{n_0+1\le i<j \le n}\!\!\!\big(x_i^2-x_j^2\big)^2 \!\!\prod_{1\le i<j \le n} \!\! \big|x_i^2-x_j^2\big|^{2\gamma}
\prod _{i=1}^{n}|x_i|^{2\alpha-1}{\rm e}^{-\frac{x_i^2}{2}}{\rm d}x
= 2^{n_1(n_1-1)+n\alpha n+ n(n-1)\gamma} \Gamma_2.
\end{gather*}
By setting $\alpha =c+1/2$ and applying the duplication formula for the gamma function, we obtain

\begin{Corollary} \label{corollary7.11} 
For $\mathrm{Re}\, c> -\frac{1}{2}$, $\mathrm{Re}\, \gamma\ge 0$, we have 
\begin{gather}
 \int_{\mathbb R^n}\prod_{i=n_0+1}^n x_i^2 \prod_{n_0+1\le i<j \le n}\big(x_i^2-x_j^2\big)^2\prod _{i=1}^{n}|x_i|^{2c}
\prod_{1\le i<j \le n} \big|x_i^2-x_j^2\big|^{2\gamma}{\rm d}\mu \notag
\\ \qquad
{} = 2^{-n_1-cn}\prod_{i=1}^{n_0}\frac{\Gamma(1+i\gamma)\Gamma(1+2c+2(i-1)\gamma)}
{\Gamma(1+\gamma)\Gamma(1+c+(i-1)\gamma)}\notag
\\ \qquad\hphantom{=}
{}\times \prod_{j=1}^{n_1}
\frac{j(1+\gamma)\Gamma(j+(n_0+j)\gamma)\Gamma(1+2c+2j+2(n_0+j-1)\gamma)}
{\Gamma(2+\gamma)\Gamma(1+c+j+(n_0+j-1)\gamma)}, \label{eq7.14}
\\
 \int_{\mathbb R^n} \prod_{n_0+1\le i<j \le n}\big(x_i^2-x_j^2\big)^2\prod _{i=1}^{n}|x_i|^{2c}
\prod_{1\le i<j \le n} \big|x_i^2-x_j^2\big|^{2\gamma}{\rm d}\mu \notag
\\ \qquad
{}= 2^{-cn}\prod_{i=1}^{n_0}\frac{\Gamma(1+i\gamma)\Gamma(1+2c+2(i-1)\gamma)} {\Gamma(1+\gamma)\Gamma(1+c+(i-1)\gamma)} \notag
\\ \qquad\hphantom{=}
{}\times \prod_{j=1}^{n_1}
\frac{j(1+\gamma)\Gamma(j+(n_0+j)\gamma)\Gamma(1+2c+2(j-1)+2(n_0+j-1)\gamma)}
{\Gamma(2+\gamma)\Gamma(c+j+(n_0+j-1)\gamma)} .\label{eq7.15}
\end{gather}
\end{Corollary}

Finally we note that the formulae~\eqref{eq7.9},~\eqref{eq7.14} and~\eqref{eq7.15} could be
recast in terms of degrees of finite reflection groups of classical type.
Let $G$ be a finite reflection groups of type $\mathrm{A}_{n-1}$,
$\mathrm{B}_n$ or $\mathrm{D}_n$ and $G_1 $ a parabolic subgroup of $G$
of the same type $\mathrm{A}_{n_1-1}$, $\mathrm{B}_{n_1}$ or $\mathrm{D}_{n_1}$.
Let $P(x)$ the product of all the normalized defining polynomials of reflecting hyperplanes:
 \begin{equation*}
 P(x) =\prod_{i=1}^N (a_{i1}x_1+\cdots+a_{in}x_n), \qquad
 \sum_{j=1}^n a_{ij}^2=2,
 \end{equation*}
 where $N$ is the number of reflecting hyperplanes of $G$. The polynomial $P_1(x)$ is defined
 for the subgroup $G_1$ similarly. Let $d_1,d_2,\dots, d_n$ be the degrees of basic invariants of
 $G$, i.e., $d_i=i$ and $ d_i=2i$, $i=1,\dots,n $ for $\mathrm{A}_{n-1}$ and $\mathrm{B}_n$ respectively,
 and $d_i=2i$, $i=1,\dots, n-1,\, d_n=n$ for $\mathrm{D}_n$ and $d_1^{(1)} ,d_2^{(1)},
 \dots, d_{n_1}^{(1)}$ be the degrees for $G_1$. Then we obtain
 \begin{Proposition} \label{proposition7.12} 
 For $ \mathrm{Re}\,\gamma \ge 0 $, we have
 \begin{equation*}
 \int_{\mathbb R^n}P_1(x)^2 |P(x)|^{2\gamma}{\rm d}\mu
=\prod_{i=1}^{n_0}\frac{\Gamma(1+d_i\gamma )}{\Gamma(1+\gamma)}
\prod_{j=1}^{n_1} \frac{j(1+\gamma)}{j+(n_0+j)\gamma}
\frac{\Gamma\big(1+d_j^{(1)}+d_{n_0+j} \gamma\big)}{\Gamma(2+\gamma)}.
\end{equation*}
\end{Proposition}

\begin{proof}
The $(\mathrm{A}_{n-1},\mathrm{A}_{n_1-1})$ case is immediate from~\eqref{eq7.9}.
The cases $(\mathrm{B}_n,\mathrm{B}_{n_1})$ and
$(\mathrm{D}_n,\mathrm{D}_{n_1})$ follow from~\eqref{eq7.14} and~\eqref{eq7.15} by
setting $c=\gamma$ and $c=0$ respectively.
\end{proof}
\begin{Remark} \label{remark7.13} 
It would be interesting to see if Proposition~\ref{proposition7.12} holds for other pairs
of finite reflection groups and their parabolic subgroups. At present we have no results in this direction.
\end{Remark}

\subsection*{Acknowledgements}
The author is very grateful to the referees and the editor for many valuable comments and suggestions.

\pdfbookmark[1]{References}{ref}
\LastPageEnding

\end{document}